\documentclass[11pt,a4paper]{article}

\usepackage[hidelinks]{hyperref}
\usepackage{amssymb, amsmath, graphicx, subcaption}
\usepackage{amsthm}
\usepackage[ruled]{algorithm2e}
\usepackage{multirow}
\usepackage{float}
\usepackage{fancyhdr}

\usepackage{geometry}
\newtheorem{theorem}{Theorem}

\fancypagestyle{firststyle}
{
	\fancyhf{}
	\fancyhead[RE,LO]{}
	\fancyfoot[RE,LO]{}
	\fancyfoot[LE,RO]{Page \thepage}
}

\begin{document}
	
	\title{Solving ordinary differential equations \\ using Schur decomposition}
	
	\author{David Arnas\thanks{Purdue University, West Lafayette, IN 47907, USA. Email: \textsc{darnas@purdue.edu}}}
	
	\date{}	
	
	\maketitle
	
	\thispagestyle{firststyle}
	
	\begin{abstract}
		This work introduces a methodology to solve ordinary differential equations using the Schur decomposition of the linear representation of the differential equation. This is done by first transforming the system into an upper triangular system using the Schur decomposition, and second, by generating the solution sequentially following the upper triangular structure. In addition, and when dealing with non-linear perturbed systems, this work proposes a methodology based on operator theory to find an approximate linear representation of perturbed non-linear systems. Particularly, we focus on polynomial differential equations and the use of Legendre polynomials to represent the solution. Based on these results, a perturbation technique is also proposed to study these problems. A set of algorithms to automate these methodologies are included.
	\end{abstract}

\section{Introduction}

Ordinary differential equations are of extreme importance in a wide variety of fields, including Physics, Mathematics, Engineering, or Economics. Of special  interest is the study of the so called Cauchy initial value problem due to its importance in dynamical systems. To that end, a large variety of approaches have been proposed over the years, including both analytical~\cite{book,var} and numerical methodologies~\cite{jorba,TFC}.

In this work we propose a technique to generate the solution of ordinary linear differential equations based on the Schur decomposition~\cite[Chapter~7]{schur}. This allows to perform a unitary transformation of the initial set of variables that transforms the matrix representation of the system into an upper triangular system that can be solved sequentially. To that end, a theorem and its corresponding algorithm are proposed to generate the sequential analytic solution of this upper triangular system of differential equations.

Unfortunately, this proposed methodology is only applicable to linear systems, and thus, a different approach is required when dealing with non-linear differential equations. Particularly, a methodology based on operator theory is proposed to obtain an approximated linear representation of perturbed problems, that is later used with the Schur decompostion technique allowing the generation of an approximate analytical solution of the system. 

Operator theory has its roots on functional analysis~\cite[Chapter~2]{conway2019course} and the study of linear operators in functional spaces~\cite[Chapter~4]{kowalski1991nonlinear}. Operator theory is commonly used in theoretical physics~\cite[Chapter~7]{naylor2000} (specially in quantum mechanics~\cite[Chapter~1]{prigogine}), and other applications~\cite{williams2015data} such as fluid-dynamics~\cite{mezic2013analysis} or control~\cite{brunton2016koopman,surana2016koopman}. Particularly, this work is based on the idea behind the Koopman operator~\cite{koopman1931hamiltonian,neumann1932operatorenmethode}, a linear operator defined in an infinite Hilbert space that is able to transform a non-linear system in a finite number of dimensions into a linear system in an infinite number of dimensions. However, since it is not possible to work on an infinite number of dimensions, we can only study a subspace, and thus, the resultant representation is an approximation of the original non-linear system. Nevertheless, this transformation provides an accurate linearization of the system (being the accuracy dependent on the problem studied and the size of the subspace selected) that can be used in many applications due to its linear properties. 

The Koopman operator assumes the existence of a set of eigenfunctions that can represent the solution of the system of differential equations. However, this is only true when the resultant system is diagonalizable, and thus, it is possible to apply the spectral theorem. This limits the direct use of the Koopman operator in many problems. To overcome this limitation, we make use of the unitary transformation provided by the Schur decomposition, which allows to solve a wider variety of problems that cannot be solved by the only use of the spectral theorem. In addition, and since the proposed technique does not have this limitation, it enables the use of a wider variety of basis functions to represent the solution of the system. In that regard, we propose an algorithm to obtain the closed form computation of the linearization based on Legendre polynomials as basis functions applied to polynomial differential equations. 

Moreover, and based on these results, a methodology is proposed to study perturbation problems, that is, systems of differential equations that can be decomposed into a strong linear part that is perturbed by a weak non-linear term. To that end, we show that the structure of the solution provided by the Schur decomposition for the unperturbed problem can be used to approximate the perturbed system. 

This manuscript is organized as follows. First, an algorithm is proposed for the automatic analytic solution generation of linear ordinary differential equations. Second, this result is extended to non-linear perturbed systems using operator theory to obtain an approximate linear representation of the system of differential equations. Additionally, a set of algorithms to obtain a closed form approximated analytical solution are provided for the case of non-linear polynomial differential equations where the basis functions used to represent the system are orthonormal Legendre polynomials in multiple dimensions. Third, we use these results to propose a methodology to study perturbed problems. And finally, we present some examples of application of these methodologies to show their performance. 


\section{Autonomous linear systems}\label{sec:linear}

In this section we focus on the solution of linear ordinary differential equations. To solve them, an algorithm based on the Schur decomposition is proposed. In particular, from the matrix representation of the linear system, this work makes use of the Schur decomposition to find an unitary transformation that allows to transform the system of differential equations into an upper triangular matrix. Once this is done, an algorithm is proposed to solve sequentially this system while providing the analytical solution of the differential equation. 

\begin{theorem} \label{theorem:schur_linear}
	Let a system of $n$ autonomous ordinary linear differential equations be defined by the Cauchy problem:
	\begin{equation}
		\left\{\begin{tabular}{l}
			$\displaystyle\frac{d\mathbf{y}}{dx}=A\mathbf{y}$ \\
			$\mathbf{y}(x_0) = \mathbf{y_0}$
		\end{tabular}\right.,
	\end{equation}
	where $\mathbf{y}=(y_1,y_2,\dots,y_n)^t$ are the functions defining the dynamic, $x$ is the independent variable, and $A$ is a $n\times n$ constant matrix. Then, by using the Schur decomposition applied over $A$ such that $A = VTV^{-1}$, where $T$ is an upper triangular matrix and $V$ is a unitary transformation, the solution of this system is:
	\begin{equation}
		\mathbf{y} = V\mathbf{\Phi},
	\end{equation}
	where $\mathbf{\Phi}= (\phi_1,\phi_2,\dots,\phi_n)^t$ is provided by:
	\begin{equation}
		\phi_i = C_ie^{T_{ii}x}+\sum_{j=i+1}^{n}r_{ij}\phi_j \quad \forall i\in\{1,\dots,n\},
	\end{equation}
	being $C_i$ the integration constants, where:
	\begin{equation}
		C_i = \displaystyle\frac{1}{e^{T_{ii}x_0}}\left[\sum_{j=1}^{n}\left(V^{-1}_{ij}y_{j}(x_0)\right) - \sum_{j=i+1}^{n}\left(r_{ij}(x_0)\phi_j(x_0)\right)\right],
	\end{equation}
	and $r_{ij}$ a set of polynomials in the form:
	\begin{equation}
		r_{ij} = \sum_{k=0}^{n-1}p_{ijk}x^k,
	\end{equation}
	whose coefficients $p_{ijk}$ are given by:
	\begin{equation}
		p_{ijk} = \left\{\begin{tabular}{ll}
			$-\displaystyle\frac{\sum_{m=i+1}^{j-1}T_{mj}p_{im(n-1)}}{T_{jj}-T_{ii}}$ & if $k = n-1$, and $T_{ii} \neq T_{jj}$ \\
			$-\displaystyle\frac{(k+1)p_{ij(k+1)}}{T_{jj}-T_{ii}}-\frac{\sum_{m=i+1}^{j-1}T_{mj}p_{imk}}{T_{jj}-T_{ii}}$ & if $k\in\{1,\dots,n-2\}$, and $T_{ii} \neq T_{jj}$ \\
			$\displaystyle\frac{T_{ij}-p_{ij1}}{T_{jj}-T_{ii}}-\frac{\sum_{m=i+1}^{j-1}T_{mj}p_{im0}}{T_{jj}-T_{ii}}$ & if $k = 0$, and $T_{ii} \neq T_{jj}$ \\
			$-\displaystyle\frac{\sum_{m=i+1}^{j-1}T_{mj}p_{im(k-1)}}{k}$ & if $k\in\{2,\dots,n-1\}$, and $T_{ii} = T_{jj}$ \\
			$T_{ij}- \sum_{m=i+1}^{j-1}T_{mj}p_{im0}$ & if $k = 1$, and $T_{ii} = T_{jj}$ \\
			$0$ & if $k = n-1$, and $T_{ii} = T_{jj}$
		\end{tabular}\right..
	\end{equation}
\end{theorem}

\begin{proof}
	The system of differential equations to solve consist on $n$ autonomous ordinary linear differential equations in the form:
	\begin{equation}\label{eq:dxax}
		\left\{\begin{tabular}{l}
			$\displaystyle\frac{d\mathbf{y}}{dx}=A\mathbf{y}$ \\
			$\mathbf{y}(x_0) = \mathbf{y_0}$
		\end{tabular}\right.,
	\end{equation}
	where $\mathbf{y}=(y_1,y_2,\dots,y_n)^t$ is a vector containing all the functions representing the solution of the differential equation, $x$ is the independent variable, $\mathbf{y_0}$ is the set of initial conditions, and $A$ is a constant matrix of size $n\times n$. In particular, each differential equation of the system can be rewritten as:
	\begin{equation}
		\displaystyle\frac{dy_i}{dx} = \sum_{j=1}^{n}A_{ij}y_j,
	\end{equation}
	and since $A_{ij}$ are bounded constants and the equation is homogeneous, the solution is Lipschitz continuous. Therefore, by applying the  Picard–Lindelöf theorem, the solution of the system exist and is unique. 
	
	The objective now is to find a process in which the system of equations described by Eq.~\eqref{eq:dxax} can be solved sequentially. In order to do that, we make use of the Schur decomposition. Particularly, the matrix $A$ can always be Schur-decomposed into:
	\begin{equation}\label{eq:schura}
		A = VTV^{-1}
	\end{equation} 
	where $T$ is an upper triangular matrix that contains in its diagonal the eigenvalues of $A$, and $V$ is an orthonormal matrix representing the unitary transformation between $A$ and $T$. It is important to note that first, the Shcur decomposition is not unique in general, and thus there can be multiple combinations of $\{T,V\}$ that decompose matrix $A$, and second, that $V^{-1}$ is always well defined since $V$ is a unitary matrix. Let $\Phi$ be a new set of variables generated by a linear combination of $\mathbf{y}$ such that $\Phi=V^{-1}\mathbf{y}$. Then, the derivative of $\Phi$ is:
	\begin{equation}
		\displaystyle\frac{d\Phi}{dx} = \frac{d}{dx}\left(V^{-1}\mathbf{y}\right) = V^{-1}\displaystyle\frac{d\mathbf{y}}{dx},
	\end{equation}
	since $V$ is a constant matrix obtained from the Schur decomposition of $A$. Moreover, by introducing Eq.~\eqref{eq:dxax} in the previous expression:
	\begin{equation}
		\displaystyle\frac{d\Phi}{dx} =  V^{-1}\displaystyle\frac{d\mathbf{y}}{dx} = V^{-1}A\mathbf{y},
	\end{equation}
	and using the Schur decomposition defined in Eq.~\eqref{eq:schura}:
	\begin{equation}
		\displaystyle\frac{d\Phi}{dx} = V^{-1}A\mathbf{y} = V^{-1}VTV^{-1}\mathbf{y} = TV^{-1}\mathbf{y},
	\end{equation}
	where we can identify $\Phi = V^{-1}\mathbf{y}$ to obtain:
	\begin{equation}
		\displaystyle\frac{d\Phi}{dx} = TV^{-1}\mathbf{y} = T\Phi,
	\end{equation}
	which is now a system of linear differential equations in $\Phi$, with $\Phi(x_0) = V^{-1}\mathbf{y_0}$, that can be solved sequentially starting with equation $n$. Therefore, from this point, we can focus on the system of differential equations:
	\begin{equation}\label{eq:dphi}
		\left\{\begin{tabular}{l}
			$\displaystyle\frac{d\Phi}{dx}=T\Phi$ \\
			$\Phi(x_0) = V^{-1}\mathbf{y_0}$
		\end{tabular}\right.,
	\end{equation}	
	and once the solution of $\Phi$ is obtained, $\mathbf{y}$ is derived using the unitary transformation from the Schur decomposition (Eq.~\eqref{eq:schura}). 
	
	For a given $\phi_i$ with $i\in\{1,\dots,n\}$, the associated differential equation can be expressed as:
	\begin{equation}\label{eq:dphij}
		\displaystyle\frac{d\phi_i}{dx}=\sum_{j=1}^{n}T_{ij}\phi_j = \sum_{j=i}^{n}T_{ij}\phi_j,
	\end{equation}
	since $T$ is an upper triangular matrix and thus $T_{ij}=0$ $\forall i>j$. We define a solution in the form:
	\begin{equation}\label{eq:proposed}
		\phi_i = C_ie^{T_{ii}x}+\sum_{j=i+1}^{n}r_{ij}\phi_j,
	\end{equation}
	where $C_i$ is the constant of integration and $r_{ij}$ is a general function of $x$. Therefore, the derivative of the proposed solution is:
	\begin{equation}
		\displaystyle\frac{d\phi_i}{dx} = T_{ii}C_ie^{T_{ii}x} + \sum_{j=i+1}^{n}\left(\frac{dr_{ij}}{dx}\phi_j+r_{ij}\frac{d\phi_j}{dx}\right).
	\end{equation}
	By multiplying Eq.~\eqref{eq:proposed} by $T_{ii}$ and rearranging the terms, the following expression is obtained:
	\begin{equation}
		T_{ii}C_ie^{T_{ii}x} = T_{ii}\phi_i - T_{ii}\sum_{j=i+1}^{n}r_{ij}\phi_j,
	\end{equation}
	which introduced in the previous equation leads to:
	\begin{equation}
		\displaystyle\frac{d\phi_i}{dx} = T_{ii}\phi_i - T_{ii}\sum_{j=i+1}^{n}r_{ij}\phi_j + \sum_{j=i+1}^{n}\left(\frac{dr_{ij}}{dx}\phi_j+r_{ij}\frac{d\phi_j}{dx}\right).
	\end{equation}
	Also, from Eq.~\eqref{eq:dphij} we know that:
	\begin{equation}
		\displaystyle\frac{d\phi_j}{dx} = \sum_{m=j}^{n}T_{jm}\phi_m,
	\end{equation}
	and thus:
	\begin{eqnarray}
		\displaystyle\frac{d\phi_i}{dx} & = & T_{ii}\phi_i - T_{ii}\sum_{j=i+1}^{n}r_{ij}\phi_j + \sum_{j=i+1}^{n}\left(\frac{dr_{ij}}{dx}\phi_j+r_{ij}\sum_{m=j}^{n}T_{jm}\phi_m\right) = \nonumber \\
		& = & T_{ii}\phi_i + \sum_{j=i+1}^{n}\left[\frac{dr_{ij}}{dx}\phi_j+r_{ij}\left(-T_{ii}\phi_j+\sum_{m=j}^{n}T_{jm}\phi_m\right)\right].
	\end{eqnarray}
	This expression can be combined with Eq.~\eqref{eq:dphij} to obtain:
	\begin{equation}
		\displaystyle\frac{d\phi_i}{dx} = \sum_{j=i}^{n}T_{ij}\phi_j = T_{ii}\phi_i + \sum_{j=i+1}^{n}\left[\frac{dr_{ij}}{dx}\phi_j+r_{ij}\left(-T_{ii}\phi_j+\sum_{m=j}^{n}T_{jm}\phi_m\right)\right],
	\end{equation}
	from which an identification of terms can be performed since the expression must hold true for the whole domain. Particularly, we can rearrange the previous expression into:
	\begin{equation}
		T_{ii}\phi_i + \sum_{j=i+1}^{n}T_{ij}\phi_j = T_{ii}\phi_i + \sum_{j=i+1}^{n}\left[\displaystyle\frac{dr_{ij}}{dx}+\sum_{m=i+1}^{n}\left(r_{im}T_{mj}\right)-T_{ii}r_{ij}\right]\phi_j,
	\end{equation}
	or, in other words:
	\begin{equation}
		T_{ij} = \displaystyle\frac{dr_{ij}}{dx}+\sum_{m=i+1}^{n}\left(r_{im}T_{mj}\right)-T_{ii}r_{ij}.
	\end{equation}	
	Moreover, since $T_{mj}=0$ $\forall m>j$, there are terms in the summation that do not contribute to the former expression and thus:
	\begin{equation}\label{eq:difrij}
		T_{ij} = \displaystyle\frac{dr_{ij}}{dx}+\sum_{m=i+1}^{j}\left(r_{im}T_{mj}\right)-T_{ii}r_{ij},
	\end{equation}
	which as can be seen, it is a linear differential equation in $r_{ij}$ that only depends on the values of $r_{im}$ with $i+1<m<j$, that is, the solution to $r_{ij}$ can be obtained sequentially. To that end, a solution to Eq.~\eqref{eq:difrij} in polynomial form is proposed such that:
	\begin{equation}
		r_{ij} = \sum_{k=0}^{n-1}p_{ijk}x^k.
	\end{equation}
	This proposed solution is then introduced in Eq.~\eqref{eq:difrij} leading to the following expression:
	\begin{equation}
		T_{ij} = \sum_{k=1}^{n-1}\left(kp_{ijk}x^{k-1}\right) + \sum_{m=i+1}^{j}\left[T_{mj}\sum_{k=0}^{n-1}\left(p_{imk}x^k\right)\right] - T_{ii}\sum_{k=0}^{n-1}\left(p_{ijk}x^k\right),
	\end{equation}
	which defines the relations between the polynomial coefficients of $r_{ij}$ ($p_{ijk}$) with the elements of the matrix $T$. This relation can be obtained by identification of the coefficients of this expression for the different powers of $x$. Particularly, for $x^0$ the following relation is obtained:
	\begin{equation}\label{eq:coef_1}
		T_{ij} = p_{ij1} + \left(T_{jj} - T_{ii}\right)p_{ij0} + \sum_{m=i+1}^{j-1}\left(T_{mj}p_{im0}\right).
	\end{equation}
	If instead, a general $x^k$ with $k\in\{1,\dots,n-2\}$ is considered, the resultant expression is:
	\begin{equation}\label{eq:coef_2}
		0 = (k+1)p_{ij(p+1)} + \left(T_{jj} - T_{ii}\right)p_{ijk} + \sum_{m=i+1}^{j-1}\left(T_{mj}p_{imk}\right).
	\end{equation}
	And finally, for the case $x^{n-1}$ the resultant equation is:
	\begin{equation}\label{eq:coef_3}
		0 = \left(T_{jj} - T_{ii}\right)p_{ij(n-1)} + \sum_{m=i+1}^{j-1}T_{mj}p_{im(n-1)}.
	\end{equation}
	
	Equations~\eqref{eq:coef_1},~\eqref{eq:coef_2}, and~\eqref{eq:coef_3} show that the structure of the solution is dependent on the relations between $T_{ii}$ and $T_{jj}$, that is, the eigenvalues of the original matrix $A$. That way, two sets of solutions can be identified:
	\begin{itemize}
		\item If $T_{ii} \neq T_{ij}$:
		\begin{equation}
			p_{ijk} = \left\{\begin{tabular}{ll}
				$-\displaystyle\frac{\sum_{m=i+1}^{j-1}\left(T_{mj}p_{im(n-1)}\right)}{T_{jj} - T_{ii}}$ & if $k = n-1$ \\
				$-\displaystyle\frac{(k+1)p_{ij(k+1)}}{T_{jj} - T_{ii}} - \frac{\sum_{m=i+1}^{j-1}\left(T_{mj}p_{imk}\right)}{T_{jj} - T_{ii}}$ & if $k\in\{1,\dots,n-2\}$ \\
				$\displaystyle\frac{T_{jj}-p_{ij1}}{T_{jj} - T_{ii}} -\frac{\sum_{m=i+1}^{j-1}\left(T_{mj}p_{im0}\right)}{T_{jj} - T_{ii}}$ & if $k = 0$
			\end{tabular}\right.
		\end{equation}
		\item If $T_{ii} = T_{ij}$:
		\begin{equation}
			p_{ijk} = \left\{\begin{tabular}{ll}
				$-\displaystyle\frac{\sum_{m=i+1}^{j-1}\left(T_{mj}p_{im(k-1)}\right)}{k}$ & if $k\in\{2,\dots,n-1\}$ \\
				$T_{ij} -\sum_{m=i+1}^{j-1}T_{mj}p_{im0}$ & if $k = 1$ \\
				$0$ & if $k = 0$
			\end{tabular}\right.
		\end{equation}
	\end{itemize}
	
	With the values $p_{ijk}$ obtained, the set of solutions $\Phi$ are completely defined but for the constants of integration $C_i$. These are provided by the initial condition in Eq.~\eqref{eq:dphi}, $\Phi(x_0) = V^{-1}\mathbf{y_0}$, combined with the proposed solution structure from Eq.~\eqref{eq:proposed}. Particularly, the constants of integration are given by the following expression:
	\begin{equation}
		C_i = \displaystyle\frac{1}{e^{T_{ii}x_0}}\left[\sum_{j=1}^{n}\left(V^{-1}_{ij}y_{j}(x_0)\right) - \sum_{j=i+1}^{n}\left(r_{ij}(x_0)\phi_j(x_0)\right)\right].
	\end{equation}	
	Once this is done, the only step left is to obtain the solution in terms of the original variables $\mathbf{y}$. This can be performed using the matrix transformation from the Schur decomposition $V$: $\mathbf{y} = V\Phi$.
\end{proof}

\begin{algorithm}[!h]
{\footnotesize
	\vspace{2mm}
	{\bf function} $p$ = p\_coefficients$(T,n,\epsilon)$ \\
	$p = zeros(n,n,n)$; \\
	\For{$i$ from $ns$ to $1$ in steps of $-1$}{
		\For{$j$ from $i+1$ to $n$}{
			\uIf{$|T(j,j)-T(i,i)| < \epsilon$}{
				\For{$s$ from $n-1$ to $1$ in steps of $-1$}{
					\For{$k$ from $i+1$ to $j-1$}{
						$p(i,j,s+1) = p(i,j,s+1) + T(k,j)*p(i,k,s)$;
					}
					$p(i,j,s+1) = - p(i,j,s+1)/s$;
				}
				\For{$k$ from $i+1$ to $j-1$}{
					$p(i,j,2) = p(i,j,2) + T(k,j)*p(i,k,1)$;
				}
				$p(i,j,2) = T(i,j)- p(i,j,2)$; \\
				$p(i,j,1) = 0$;
			}
			\Else{
				\For{$k$ from $i+1$ to $j-1$}{
					$p(i,j,n) = p(i,j,n) + T(k,j)*p(i,k,n)$;
				}
				$p(i,j,n) = - p(i,j,n)/(T(j,j)-T(i,i))$; \\
				\For{$s$ from $n-2$ to $1$ in steps of $-1$}{
					\For{$k$ from $i+1$ to $j-1$}{
						$p(i,j,s+1) = p(i,j,s+1) + T(k,j)*p(i,k,s+1)$;
					}
					$p(i,j,s+1) = -((s+1)*p(i,j,s+2)+p(i,j,s+1))/(T(j,j)-T(i,i))$;
				}
				\For{$k$ from $i+1$ to $j-1$}{
					$p(i,j,1) = p(i,j,1) + T(k,j)*p(i,k,1)$;
				}
				$p(i,j,1) =  T(i,j) - p(i,j,2) - p(i,j,1)$; \\
				$p(i,j,1) = p(i,j,1)/(T(j,j)-T(i,i))$;
				
			}
		}
	}
	{\bf end}
	\caption{Coefficients $p_{ijk}$} \label{alg:Schur}
}
\end{algorithm}

Algorithm~\ref{alg:Schur} presents a subroutine to obtain the coefficients $p_{ijk}$ from the triangular matrix $T$ using the result from Theorem~\ref{theorem:schur_linear}. In here, a parameter $\epsilon$ has been included to control the potential numerical errors resultant from the Schur decomposition and the comparison of eigenvalues of matrix $A$.

It is important to note that since the decomposition of a matrix in its Jordan normal form $J$ and its respective matrix of generalized eigenvectors $E$ is a particular case of Schur decomposition, this theorem can also be applied to a Jordan normal form. This has its advantages and disadvantages. The main advantage is that since the Jordan normal form minimizes the number of terms out of the diagonal in the upper triangular matrix $T$, so does the number of operations to obtain the solution of the differential equation. However, it presents the disadvantage of having to obtain the Jordan normal form itself, which, in general, is a computationally more expensive process and its calculation is very sensitive to perturbations. 


\section{Non-linear perturbed systems} \label{sec:nonlinear}

The results from Section~\ref{sec:linear} are only applicable to autonomous ordinary linear differential equations, and thus, they cannot be directly generalized to non-linear systems. However, it is still possible to obtain an approximation of perturbed non-linear systems using a linear representation. 

In this work, we approach this problem by performing an expansion of the dimensionality of the system and projecting the space into a set of orthonormal basis functions that will act as the new variables of the system of differential equations. This is the idea behind operator theory techniques such as the Koopman Operator. In general, applying these techniques to general problems causes the resultant matrices to be non-diagonalizable and thus, it is not possible to apply the spectral theorem or to effectively find the eigenfunctions of the system. This is a limitation that the Schur decomposition technique presented before does not share, allowing the effective application of this methodology to a wider set of problems.

In this section we present a methodology to linearize perturbed non-linear systems using operator theory. To that end, we present a general case where no assumption on the basis functions selected or the structure of the differential equations is performed. After that, we focus on the particular case of using Legendre polynomials applied to polynomial differential equations. A set of algorithms to automate this process are also included. These methodologies are specially suited for quasi-periodic systems and perturbation problems, although they can also be used in other problems but with more limitations. 

\subsection{General case}

Let the non-linear system of $n$ ordinary differential equations in study be defined by:
\begin{equation}
	\left\{\begin{tabular}{l}
		$\displaystyle\frac{d\mathbf{y}}{dx}=\mathbf{f}(\mathbf{y},x)$ \\
		$\mathbf{y}(x_0) = \mathbf{y_0}$
	\end{tabular}\right.,
\end{equation}
where $\mathbf{y}$ is the set of dependent variables that are required to be solved, $x$ is the independent variable, and $\mathbf{f}(\mathbf{y},x)$ is a function of the previous variables. In general, $\mathbf{f}(\mathbf{y},x)$ is a non-linear function that depends both in the dependent and independent variables. The goal then is to first make the system autonomous, if it is not already, and second, to find an approximate linear system that is able to represent the differential equation with an adequate accuracy.  

The transformation to an autonomous system can be done by introducing a new variable $z$ such that its value evolves linearly with the independent variable, that is, $z = \tau x$ and $dz/dx = \tau$ for all $x$, where $\tau$ is a constant that allows to normalize the time evolution within the ranges of definition of the basis functions. This allows to rewrite the original system as an autonomous system in the form of:
\begin{equation}
	\left\{\begin{tabular}{l}
		$\displaystyle\frac{d\mathbf{\tilde{y}}}{dx}=\mathbf{\tilde{f}}(\mathbf{\tilde{y}})$ \\
		$\mathbf{\tilde{y}}(x_0) = \mathbf{\tilde{y}_0}$
	\end{tabular}\right.,
\end{equation}
where $\mathbf{\tilde{y}} = \{z,\mathbf{y}\}$,  $\mathbf{\tilde{y}}(x_0) = \{\tau x_0,\mathbf{y}\}$ and $\mathbf{\tilde{f}}(\mathbf{\tilde{y}}) = \{\tau, \mathbf{f}(\mathbf{y},z)\}$. Note that, this is effectively expanding the system into an additional dimension with $n+1$ differential equations. From this point, we will only consider elements $\mathbf{y}$ and $\mathbf{f}(\mathbf{y})$ to represent the resultant autonomous system to simplify the notation in the following sections of the manuscript.

The objective now is to find an approximate linear representation of the differential equation. 
This is done in four steps. First, a normalization is performed in the variables of study to make them range between $[-1,1]$. This allows higher order polynomials to be systematically smaller than lower order monomials, which enables the convergence of the methodology for a finite number of terms. Second, an expansion on the number of dependent variables is performed. This allows to represent part of the non-linearities of the system of differential equations. Third, this expanded set of variables is transformed into a set of orthonormal basis functions that are used to represent the system. And finally, the Garlekin method is applied to obtain the projection of the resultant differential equation into the selected set of orthonormal basis functions while minimizing the error. This leads to a final system whose variables are the set of basis functions selected and that is completely linear. 

Let $m$ be the number and $\mathbf{h}=\mathbf{h}(\mathbf{y})$ with $\mathbf{h}=(h_1,h_2,\dots,h_m)^t$ be the set of orthonormal basis functions that are selected to represent the expanded system of differential equations. These basis functions, in order to be orthonormal, have to fulfill the following condition:
\begin{equation}
	\langle  h_i, h_j \rangle = \int_{\Omega} \rho h_ih_jd\Omega = \left\{\begin{tabular}{ll}
		0 & if $i \neq j$ \\
		1 & if $i = j$
	\end{tabular}\right. \quad \forall i,j \in \mathbb{N}^m,
\end{equation}
where $\langle  h_i, h_j \rangle$ is the inner product between $h_i$ and $h_j$, $\Omega$ is the domain of definition of the basis functions, and $\rho$ is the weighting function associated with the set of basis functions selected. Examples of these basis functions include, for instance, any combination of Fourier series, Hermite polynomials, Legendre polynomials, or Chebyshev polynomials. The methodology can also be used with non orthonormal basis functions. However, this requires more computations since all the inner products have to be normalized due to the use of a non-orthonormal space. For this reason, we focus on the use of orthonormal basis functions.

Then, by using these basis functions as the new set of variables, the following system of $m$ equations is obtained:
\begin{equation}
	\displaystyle\frac{d\mathbf{h}}{dx} = \nabla\mathbf{h}\mathbf{f},
\end{equation}  
where:
\begin{equation}
	\nabla \mathbf{h} = \left(\begin{array}{ccc} 
		\displaystyle\frac{\partial }{\partial y_1}h_1({\bf y}) & \cdots & \displaystyle\frac{\partial }{\partial y_n}h_1({\bf y}) \\ 
		\vdots & \ddots & \vdots\\
		\displaystyle\frac{\partial }{\partial y_1}h_m({\bf y}) & \cdots& \displaystyle\frac{\partial }{\partial y_n}h_m({\bf y})\end{array}\right),
\end{equation}
is a matrix of size $m \times n$ representing the gradient of the basis functions $\mathbf{h}$. This system is still non-linear due to the non-linearities present in $\mathbf{f}$. However it can be linearized by projecting it back into the set of basis functions already defined. This is the reason why we need to normalize the variables, to impose that higher order monomials make a smaller contribution in the solution of the system. 

Let $M$ be the matrix containing all the coefficients associated with this projection, that is, $M$ is the matrix of size $m\times m$ containing all the combinations of inner products that can be generated from the former expression. In particular:
\begin{equation} \label{eq:matrix}
	M_{ij} = \langle  \nabla h_i\mathbf{f}, h_j \rangle = \int_{\Omega} \rho \nabla h_i\mathbf{f} h_j d\Omega = \int_{\Omega}\rho \left[ \sum_{k=1}^{n} \displaystyle\frac{\partial h_i}{\partial y_{k}} f_k\right] h_j d\Omega,
\end{equation} 
where the components of $M$ can be used to approximate each total derivative as:
\begin{equation}
	\displaystyle\frac{d h_i}{dx} = \sum_{j=1}^{m} M_{ij}h_j.
\end{equation}
This allows to transform the original problem into the following system:
\begin{equation}
	\left\{\begin{tabular}{l}
		$\displaystyle\frac{d\mathbf{h}}{dx} = M\mathbf{h}$ \\
		$\mathbf{h}(x_0) = \mathbf{h}(\mathbf{y_0})$
	\end{tabular}\right.,
\end{equation}
which is now a linear system of $m$ differential equations. In order to solve it, we make use of the Schur decomposition to be able to obtain the evolution of each basis function from the system sequentially. Particularly, the Schur decomposition of $M$ is defined as $M = VTV^{-1}$, where $T$ is an upper triangular matrix whose diagonal contains the eigenvalues of $M$, and $V$ is a unitary transformation matrix. Let $\Phi = (\phi_1,\phi_2,\dots,\phi_m)^t$ be a set of new dependent variables related to the basis functions through:
\begin{equation}
	\Phi = V^{-1}\mathbf{h}.
\end{equation} 
Then, the derivative of this new set of variables is:
\begin{equation}
	\displaystyle\frac{d\Phi}{dx} = V^{-1}\frac{d\mathbf{h}}{dt} = V^{-1}M\mathbf{h} = TV^{-1}\mathbf{h} = T\Phi,
\end{equation}
which results in the following system of differential equations:
\begin{equation}
	\left\{\begin{tabular}{l}
		$\displaystyle\frac{d\Phi}{dx} = T\Phi$ \\
		$\Phi(x_0) = V^{-1}\mathbf{h}(x_0)$
	\end{tabular}\right. .
\end{equation}
This system can be solved sequentially starting by $\phi_m$, being the structure of the solution the one provided by Theorem~\ref{theorem:schur_linear}:
\begin{equation}
	\phi_i = C_ie^{T_{ii}x}+\sum_{j=i+1}^{n}r_{ij}\phi_j \quad \forall i\in\{1,\dots,n\},
\end{equation}
where $C_i$ are the integration constants and $r_{ij}$ is a known polynomial in the independent variable $(x)$.

Once the solution of $\Phi$ is obtained, it is possible to express the solution in the original set of variables. However, in general, the initial set of variables do not have an exact representation using the set of basis functions selected, and thus, a projection is again required to represent these variables. Let, $H$ be the $n\times m$ matrix containing the coefficients of that projection. Then, these coefficients are provided by this inner product:
\begin{equation} \label{eq:modes}
	H_{ij} = \langle  y_i, h_j \rangle = \int_{\Omega}\rho y_i h_j d\Omega,
\end{equation}
which means that the solution of the original variables can be represented as:
\begin{equation}
	\mathbf{y} = H\mathbf{h} = HV\Phi.
\end{equation}

It is important to note that due to the finite set of basis functions used, there are two conditions that this methodology has to fulfill in order to converge. First, the variables have to be bounded in the range $[-1,1]$. Alternatively, the weighting function $\rho$ has to make them decrease very fast outside this range. Second, the coefficients of the non-linear part cannot be larger than the linear part. These conditions are required in order to not increase the error as the order of the basis functions increases. Note that perturbation problems are one of the examples of these kind of problems if their variables are normalized properly. 

\subsection{Using Legendre polynomials as basis functions}

Legendre polynomials have some interesting properties when compared to other orthogonal basis functions. First, they are based on polynomials and thus, they have very good analytical properties and are easy to work with them. Second, they are defined in the range $[-1,1]$, making the range of the integral well defined as opposed to other basis functions such as Hermite functions. Additionally, this is the range in which we are interested since it makes the methodology to converge. Finally, the weighting function $\rho$ is constant, meaning that the weighting distribution is uniform in the whole domain of definition (and as such, the error). All these properties allow to obtain a closed form solution of the inner products from Eqs.~\eqref{eq:matrix} and~\eqref{eq:modes} when the differential equation is represented in polynomial form.  

In general, non-linear systems are not fully represented in polynomial form. However, in many cases it is possible to find a polynomial representation of the system even if that means increasing the number of variables to be able to perform that representation. An example of this is the zonal harmonics problem that appears in in the motion of an object around a non-spherical celestial body~\cite{zonal}. Therefore, in this section we will focus on differential equations that are fully polynomial, that is, that can be represented as:
\begin{equation}
	\displaystyle\frac{d y_i}{dx}=f_i(\mathbf{y}) = \sum_{j=1}^{T_i}\kappa_{(i,j)}\prod_{k=1}^{n}y_k^{\gamma_{(i,j,k)}},
\end{equation}
for all $i\in\{1,\dots,n\}$, where $T_i$ is the number of monomials in the polynomial $f_i$, $\kappa_{(i,j)}$ is the coefficient of the monomial $j$ from $f_i$, and $\gamma_{(i,j,k)}$ is the exponent of variable $y_k$ that belongs to the $j$-th monomial from the polynomial representation of the derivative of variable $y_i$. Therefore, we can represent these polynomials in an array $F$ of size $n\times nt \times (n + 1)$ where $nt$ is the maximum of all the $T_i$ $\forall i\in\mathbb{N}^n$. In this array $F(i,j,k)$, the first term identifies the differential equation $i$, the second term, each of the individual monomials represented from $f_i$, and finally, the third term relates to the coefficient $\kappa_{(i,j)}$ if $k = 1$, and to the exponents $\gamma_{(i,j,k)}$ if $k\neq 1$. In other words, we can represent the $j$-th monomial from differential equation $i$ as: $F(i,j,1) = \kappa_{(i,j)}$, $F(i,j,k+1) = \gamma_{(i,j,k)}$ $\forall k\in\mathbb{N}^n$. This process has to be continued for all the variables of the problem $y_i$ with $i\in\{1,\dots,n\}$.

Once the array $F$ is defined, we have to select the maximum order of orthonormal basis functions that we are interested to use. This is a very important parameter that will determine both the amount of computation performed (and also the size of matrix $M$), and the accuracy of the methodology. Let $\sigma$ be the maximum order of basis functions. Then, we can use the properties of combinatorics to obtain the number of basis functions that will be generated ($m$):
\begin{equation}
	m = \left(\begin{array}{c}
		n+\sigma \\
		n
	\end{array}\right),
\end{equation}
which can be translated into a subroutine like Algorithm~\ref{alg:basis_number}. 

\begin{algorithm}[!h]
{\footnotesize
	\vspace{2mm}
	{\bf function} $m$ = basis\_number$(n,\sigma)$ \\
	$m = 1$; \\
	\For{$i$ from $1$ to $n$}{
		$m = m*(\sigma+i)/i$;
	}
	{\bf end}
	\caption{Number of basis functions} \label{alg:basis_number}
}
\end{algorithm}

With $F$ and the number of basis functions $m$, we have all the elements to obtain the matrix $M$ (see Eq.~\eqref{eq:matrix}). The process to compute it is as follows. First, we define the ordering of the basis functions based on the maximum order of their monomials. Second, we generate the Legendre polynomials in one variable up to order $\sigma$ and normalize them. Third, we obtain the derivative of each one of these normalized Legendre polynomials. Fourth, we generate the multidimensional Legendre polynomials based on the one dimensional normalized Legendre polynomials and the ordering defined for the basis functions. Fifth, the computation of matrix $M$ is performed in closed form. Finally, the original variables $\mathbf{y}$ are projected also in the same set of basis functions to obtain the matrix $H$. Each one of these steps is studied in more detail in the following subsections.

\subsubsection{Ordering of basis functions}

Defining an ordering of basis functions is important as a way to control the structure of the algorithm operations. This allows to reduce the overall computational time of the algorithm. Let $I$ be an array of size $m \times n$, where $I(i,j)$ provides the order of the variable $j\in\{1,\dots,n\}$ in the multidimensional basis function $i\in\{1,\dots,m\}$. Note tat this array $I$ can also be used to determine the maximum exponent of each variable in a given basis polynomial. 

\begin{algorithm}[!h]
	{\footnotesize
		\vspace{2mm}
		{\bf function} $I$ = ordering$(n,\sigma, m)$ \\
		$I = zeros(m,n)$; \\
		$s = 1$; \\
		\For{$k$ from $1$ to $\sigma$}{
			$s = s + 1$; \\
			$I(s,1) = k$; \\
			$temp = I(s,:)$; \\
			\While{$I(s,\sigma)\neq k$}{
				$temp(1) = temp(1) + 1$; \\
				\For{$i$ from $1$ to $n-1$}{
					\uIf{$temp(i) \leq k$}{
						break \\
					}
					\Else{
						$temp(i) = 0$; \\
						$temp(i+1) = temp(i + 1) + 1$;
					}
				}
				$sum = 0$; \\
				\For{$i$ from $1$ to $n$}{
					$sum = sum + temp(i)$;
				}
				\If{$sum = k$}{
					$s = s + 1$; \\
					$I(s,:) = temp$;
				}
			}
		}
		{\bf end}
		\caption{Ordering of basis functions} \label{alg:ordering}
	}
\end{algorithm}

For the purposes of this algorithm we will order the basis functions based first on their order $\sigma$, and second, on the exponents of $y_i$ following a lexicographic ordering. Algorithm~\ref{alg:ordering} shows the subroutine to generate such ordering for an arbitrary number of dimensions and maximum order of the basis functions. This algorithm works as a digital counter in base $\sigma$, where it starts with the basis function of order $0$, that is, the constant monomial, and continues generating all the basis functions with increasing order $\sigma$. 

The advantage of using this ordering is that if later we require a larger set of basis functions, we can use the results from previous computations. The same can be said if we have a computation performed with a large set of basis functions and we are interested to reduce the number of terms of the solution. An additional advantage is that in some problems it may be interesting to skip basis functions of some order since they are not used in the solution. This happens, for instance, when the differential equation is composed by monomials of odd order. In this case, only basis functions of odd order are required to represent the differential equation.

\subsubsection{Orthonormal Legendre polynomials}

The orthonormal Legendre polynomials are generated in two steps. First, the Legendre polynomials are obtained using the recursive function:
\begin{equation}
	L_{i+1} = \displaystyle\frac{(2 i + 1) y L_i - y L_{i-1} }{i + 1},
\end{equation}
where $L_i$ is the Legendre polynomial of order $i$ in one variable, and $L_0 = 1$ and $L_{1} = y$ are the Legendre polynomials that start the iteration. After this, the polynomials are normalized with:
\begin{equation}
	N_i = \displaystyle\sqrt{\displaystyle\frac{2i + 1}{2}} L_i,
\end{equation}
where $N_i$ are the normalized Legendre polynomials. This normalization assures that:
\begin{equation}
	\langle  N_i, N_i \rangle = 1 \quad \forall i.
\end{equation}

\begin{algorithm}[!h]
	{\footnotesize
		\vspace{2mm}
		{\bf function} $L$ = Legendre$(\sigma)$ \\
		$L = zeros(\sigma+1,\sigma+1)$; \\
		$L(1,1) = 1$; \\
		$L(2,2) = 1$; \\
		\For{$i$ from $3$ to $\sigma + 1$}{
			\For{$j$ from $1$ to $i-1$}{
				$L(i,j+1) = L(i,j+1) + (2*(i-2) + 1)/(i-1)*L(i-1,j)$; \\
				$L(i,j)   = L(i,j)   - (i-2)/(i-1)*L(i-2,j)$;
			}
		}
		{\bf end}
		\caption{Legendre polynomials} \label{alg:legendre}
	}
\end{algorithm}

\begin{algorithm}[!h]
	{\footnotesize	
		\vspace{2mm}
		{\bf function} $N$ = normalized\_Legendre$(L,\sigma)$ \\
		$N = zeros(\sigma+1,\sigma+1)$; \\
		\For{$i$ from $3$ to $\sigma + 1$}{
			\For{$j$ from $1$ to $\sigma + 1$}{
				$N(i,j) = sqrt((2*(i-1)+1)/2)*L(i,j)$;
			}
		}
		{\bf end}
		\caption{Normalized Legendre polynomials} \label{alg:normalized}
	}
\end{algorithm}

These polynomials are stored in memory in the arrays $L$ and $N$ of size $(\sigma + 1)\times (\sigma + 1)$, where $\sigma$ is the maximum order of the basis functions used. In each of the components of this array it is stored the coefficient of the monomials of the Legendre polynomials taking into account an increasing ordering of the exponents, that is:
\begin{equation}
	L_i = \sum_{j=1}^{\sigma + 1} L(i,j)y^{j-1},
\end{equation}
where $L(i,j)$ are the components of the array of coefficients. Algorithm~\ref{alg:legendre} shows the subroutine to generate the array containing the coefficients of the Legendre polynomials, and Algorithm~\ref{alg:normalized} takes this result to generate the case of normalized Legendre polynomials.

\subsubsection{Derivative of orthonormal Legendre polynomials}

The derivative of the orthonormal Legendre polynomials is also stored in an array $D$ of size $(\sigma + 1)\times (\sigma + 1)$ and follows the same structure as the case for Legendre polynomials. Algorithm~\ref{alg:derivative} shows the subroutine to obtain this array from the normalized Legendre polynomials array.

\begin{algorithm}[!h]
	{\footnotesize
	\vspace{2mm}
	{\bf function} $D$ = derivative\_Legendre$(N,\sigma)$ \\
	$D = zeros(\sigma+1,\sigma+1)$; \\
	\For{$i$ from $2$ to $\sigma + 1$}{
		\For{$j$ from $1$ to $\sigma$}{
			$D(i,j) = j*N(i,j+1)$;
		}
	}
	{\bf end}
	\caption{Derivative of normalized Legendre polynomials} \label{alg:derivative}
}
\end{algorithm}

\subsubsection{Multidimensional Legendre polynomials}

The multidimensional Legendre polynomials are generated by multiplication of the one-dimensional normalized Legendre polynomials. In here, we make use of the ordering defined in Algorithm~\ref{alg:ordering} to select the order of generation of these polynomials, and also, the order in which the monomial exponents are represented. In other words, a multidimensional Legendre polynomial $\mathcal{L}_i$ can be represented as:
\begin{equation}
	\mathcal{L}_i = \sum_{j=1}^{m}\mathcal{L}(i,j)\prod_{k=1}^{n}y_k^{I(j,k)}.
\end{equation}
Algorithm~\ref{alg:multidimensional} shows the subroutine to generate these polynomials from the set of normalized Legendre polynomials in one dimension. Only the coefficients of these multidimensional polynomials are stored into an array of size $m\times m$ where the rows represent each individual polynomial, and the columns, the specific monomials of each polynomial.

\begin{algorithm}[!h]
	{\footnotesize
	\vspace{2mm}
	{\bf function} $\mathcal{L}$ = multidimensional\_Legendre$(N,n,m)$ \\
	$\mathcal{L} = zeros(m,m)$; \\
	\For{$i$ from $2$ to $m$}{
		\For{$j$ from $1$ to $m$}{
			$\mathcal{L}(i,j) = 1$; \\
			\For{$k$ from $1$ to $n$}{
				$\mathcal{L}(i,j) = \mathcal{L}(i,j)*N(I(i,k)+1,I(j,k)+1)$;
			}
		}
	}
	{\bf end}
	\caption{Multidimensional Legendre polynomials} \label{alg:multidimensional}
}
\end{algorithm}

\subsubsection{Operator matrix}

The computation of the operator matrix $M$ requires three steps. First, to obtain the partial derivatives of the multidimensional Legendre polynomials ($\nabla\mathbf{h}$). Second, to obtain the total derivative of these same polynomials $\nabla\mathbf{h}\mathbf{f}$. And third, to perform the inner products $\langle  \nabla\mathbf{h}\mathbf{f}, \mathbf{h} \rangle$ to obtain the operator matrix $M$.

\begin{algorithm}
	{\footnotesize
		\vspace{2mm}
		{\bf function} $M$ = operator\_matrix$(\mathcal{L},N,D,n,m)$ \\
		$M = zeros(m,m)$;  $\mathcal{P} = zeros(n,m)$;  $\mathcal{F} = zeros(m*nt,n+1)$; $\mu = zeros(1,n)$; $\eta = 2^n$; \\
		\For{$i$ from $1$ to $m$}{
			Partial Derivatives of Legendre polynomials in multiple dimensions ($\nabla\mathbf{\mathcal{L}}$) \\
			\For{$j$ from $1$ to $m$}{
				\For{$k$ from $1$ to $n$}{
					$\mathcal{P}(k,j) = 1$; \\
					\For{$h$ from $1$ to $n$}{
						\uIf{$k = h$}{
							$\mathcal{P}(k,j) = \mathcal{P}(k,j)*D(I(i,h)+1,I(j,h)+1)$; \\
						}
						\Else{
							$\mathcal{P}(k,j) = \mathcal{P}(k,j)*N(I(i,h)+1,I(j,h)+1)$;
						}
					}
				}
			}
			Total derivative ($\nabla\mathbf{\mathcal{L}}*\mathbf{f}$) \\
			$s = 0$; \\
			\For{$d$ from $1$ to $n$}{
				\For{$h$ from $1$ to $nt$}{
					\uIf{$F(d,h,1)=0$}{
						break; \\					
					}
					\Else{
						\For{$j$ from $1$ to $m$}{
							\If{$\mathcal{L}(d,j)\neq 0$}{
								$s = s + 1$; \\
								$\mathcal{F}(s,1) = F(d,h,1)*\mathcal{P}(d,j)$; \\
								\For{$k$ from $2$ to $n+1$}{
									$\mathcal{F}(s,k) = F(d,h,k) + I(j,k-1)$;
								}
							}
						}
					}
				}
			}
			
			Matix integration ($M$) \\
			\For{$k$ from $1$ to $m$}{
				\For{$j$ from $1$ to $m$}{
					\If{$\mathcal{L}(k,j) \neq 0$}{
						\For{$h$ from $1$ to $s$}{
							flag $= 1$; \\
							\For{$d$ from $1$ to $n$}{
								$\mu(d) = \text{round}(I(j,d) + \mathcal{F}(h,d+1) + 1)$; \\
								\If{$\mod(\mu(d),2) = 0$}{
									flag $= 0$; \\
									break;
								}
							}
						}
						\If{flag $= 1$}{
							$\nu = 1$; \\
							\For{$d$ from $1$ to $n$}{
								$\nu = \nu*\mu(d)$;
							}
							$M(i,j) = M(i,j) + \eta*\mathcal{L}(k,j)*\mathcal{F}(h,1)/\nu$;
						}
					}
				}
			}
			
		}
		{\bf end}
		\caption{Operator matrix} \label{alg:matrix}}
\end{algorithm}

The partial derivatives are obtained by multiplying the one-dimensional normalized Legendre polynomials ($N$) and their derivatives ($D$). In here, we make use of the property that each of these polynomials is only function of a single variable, therefore:
\begin{equation}
	\displaystyle\frac{\partial h_i}{\partial y_j} = \frac{\partial}{\partial y_j}\left(\prod_{\substack{k=1}}^{n}N_{ik}(y_k).\right) = \frac{\partial N_{ij}(y_j)}{\partial y_j}\prod_{\substack{k=1 \\ k\neq j}}^{n}N_{ik}(y_k),
\end{equation}
$N_{ik}$ are the normalized Legendre polynomials in variable $y_k$ corresponding to the basis function $h_i$. Algorithm~\ref{alg:matrix} shows the subroutine to generate these partial derivatives, where only the coefficients of the polynomials are stored in an array of size $n\times m$ for each individual basis function. This array contains the $n$ partial derivatives in each one of its rows, where the resultant polynomial is stored with the same structure as the array $\mathcal{L}$.  

The next step requires obtaining the total derivative, that is, $\nabla\mathbf{h}\mathbf{f}$. This is done by obtaining each individual monomial resulting from the multiplication, and storing the resultant coefficients in an array $\mathcal{F}$. This array has size $(m*nt)\times (n+1)$ to account for all possible combinations of monomials that can be generated in the multiplication of $\mathbf{b}$ (which at most has $nt$ terms) and $\nabla\mathbf{h}$ (which in the worst case scenario has $m$ potential elements). The structure of $\mathcal{F}$ has a similar structure as the array containing the information of the differential equation $F$, that is, in $\mathcal{F}(i,j)$ the first component $i$ alludes to the monomial in study, while the second $j$ to the coefficient of the monomial if $j=1$, or to the exponent of each variable of the monomial if $j\neq 1$. In other words, the total derivative of the basis function $h_i$ is represented as:
\begin{equation}
	\displaystyle\frac{d h_i}{dx} = \sum_{j=1}^{m*nt}\mathcal{F}(j,1)\prod_{k=2}^{n}y_{k-1}^{\mathcal{F}(j,k)}.
\end{equation}
This allows to perform the multiplication by dealing separately with the coefficients and the monomial exponents as seen in Algorithm~\ref{alg:matrix}.

The last step involves performing the analytical inner product between the total derivative $dh_i/dx$ and the set of basis functions $\mathbf{h}$. This integration is performed individually for each monomial of $\mathcal{F}$ taking advantage of the properties of even and odd functions. Particularly, when performing the integral in the symmetric range $[-1,1]$ and using even functions:
\begin{eqnarray}
	\int_{\Omega} \prod_{i=1}^{n}y_i^{\gamma_i} = \displaystyle\frac{2^n}{\prod_{i=1}^{n}(\gamma_i+1)} \quad \text{if} \ \{\gamma_i = 0 \mod(2)\} \ \forall i,
\end{eqnarray}
while for odd functions the integral is:
\begin{eqnarray}
	\int_{\Omega} \prod_{i=1}^{n}y_i^{\gamma_i} = 0 \quad \text{if} \ \exists i \ni \{\gamma_i \neq 0 \mod(2)\}.
\end{eqnarray}
By doing this, the algorithm is able to avoid performing the computation of a large proportion of the required integrals, being the rest of them obtained analytically using the previous expressions. Algorithm~\ref{alg:matrix} shows the subroutine to perform these inner products, and thus obtaining the operator matrix $M$, based on the previous results.

\subsubsection{Projection of the original variables}

In general, the original variables are not expressed in the set of basis functions selected. Therefore, it is required to find and accurate representation of the original variables $\mathbf{y}$ into the basis functions $\mathbf{h}$. This is done using the Garlekin method through the computation of the inner products as seen in Eq.~\eqref{eq:modes}. However, due to the structure of  $\mathbf{y}$ and $\mathbf{h}$ in the case of Legendre polynomials, this can be done directly since all the variables in $\mathbf{y}$ are monomials of first order in one variable, and thus, they only have components in the normalized Legendre polynomials of first order. Particularly, we know that for a first order basis function: 
\begin{equation}
	\langle  a_i y_i, a_iy_i \rangle = \int_{\Omega} a_i^2y_i^2 d\Omega = \displaystyle\displaystyle\frac{2^n}{3}a_i^2 = 1,
\end{equation}
therefore, the normalized coefficient for the first order Legendre polynomials is $a_i = \sqrt{3/2^n}$. Thus, the projection is:
\begin{equation}
	\langle  y_i, \displaystyle\sqrt{\displaystyle\frac{3}{2^n}}y_i \rangle = \int_{\Omega}\sqrt{\displaystyle\frac{3}{2^n}} y_i^2 d\Omega = \displaystyle\sqrt{\displaystyle\frac{2^n}{3}},
\end{equation}
which is coherent with the result that $y_i = \sqrt{2^n/3}a_iy_i$. Algorithm~\ref{alg:initial} shows the subroutine to obtain this projection in an automated process for any number of dimensions and order of basis functions.

\begin{algorithm}[!h]
	{\footnotesize
	\vspace{2mm}
	{\bf function} $H$ = projection\_variables$(n,m)$ \\
	$H = zeros(n,m)$; \\
	$\psi = \sqrt{2^n/3}$; \\
	\For{$i$ from $1$ to $n$}{
		$H(i,i+1) = \psi$;
	}
	{\bf end}
	\caption{Projection of the original variables} \label{alg:modes}
}
\end{algorithm}

\subsubsection{Boundary conditions for the basis functions}

The last element required to solve the problem is the generation of the boundary conditions for the expanded space, that is, for the set of basis functions $\mathbf{h}$. This can be directly obtained by applying the initial conditions in $\mathbf{y}(x_0)$ into the basis functions. Algorithm~\ref{alg:initial} shows an algorithm to obtain those conditions, where $h0$ is the array of size $1\times m$ containing the boundary conditions in the set of basis functions, and $y0$ is an array of size $1\times n$ containing the boundary conditions of the original variables $\mathbf{y}$.

\newpage
\begin{algorithm}[!h]
	{\footnotesize
	\vspace{2mm}
	{\bf function} $h0$ = boundary\_Legendre$(\mathcal{L},I,n,m)$ \\
	$h0 = zeros(m,1)$; \\
	\For{$i$ from $1$ to $m$}{
		\For{$j$ from $1$ to $m$}{
			$mult = \mathcal{L}(i,j)$; \\
			\For{$k$ from $1$ to $n$}{
				$mult = mult*y0(k)^{I(j,k)}$;
			}
			$h0(i) = h0(i) + mult$;
		}
	}
	{\bf end}
	\caption{Boundary conditions for the basis functions} \label{alg:initial}
}
\end{algorithm}


\section{Perturbation model based on the Shur decomposition}

Perturbation problems are based on the idea of having a set of differential equations that can be studied as a combination of a linear system whose solution is known, plus the contribution of a non-linear term that is smaller than the linear term and that makes the system non-solvable analytically. In other words, these are problems that can be represented in this form:
\begin{equation} \label{eq:diff_ord1}
	\left\{\begin{tabular}{l}
		$\displaystyle\frac{d\mathbf{y}}{dx}=\mathbf{f}(\mathbf{y}) = \mathbf{b}(\mathbf{y}) + \epsilon\mathbf{g}(\mathbf{y}) $ \\
		$\mathbf{y}(x_0) = \mathbf{y_0}$
	\end{tabular}\right.,
\end{equation}
where $\mathbf{b}$ is a linear function in $\mathbf{y}$, $\mathbf{g}(\mathbf{y})$ is in general non-linear, and $\epsilon$ is a small parameter. The objective then is to apply the results from Sections~\ref{sec:linear} and~\ref{sec:nonlinear} to the solution and study of these systems.

Let $\mathbf{h}$ be a set of orthonormal basis functions as defined in Section~\ref{sec:nonlinear}. Then, the derivative of this set of basis functions is:
\begin{equation}
	\displaystyle\frac{d\mathbf{h}}{dx} = \nabla\mathbf{h}\mathbf{f} = \nabla\mathbf{h}\mathbf{b} + \epsilon\nabla\mathbf{h}\mathbf{g}.
\end{equation}
As proposed in the previous section, and if the variables are properly normalized, this system can be linearized in the set of basis functions $\mathbf{h}$ using the Garlekin method. That way, it is possible to transform the previous system into this approximated linear system:
\begin{equation}\label{eq:hpert}
	\left\{\begin{tabular}{l}
		$\displaystyle\frac{d\mathbf{h}}{dx} = M\mathbf{h} = B\mathbf{h} + \epsilon P\mathbf{h}$ \\
		$\mathbf{h}(x_0) = \mathbf{h}(\mathbf{y_0})$
	\end{tabular}\right.,
\end{equation}
where $M$, $B$, and $P$ are defined as the following inner products:
\begin{eqnarray}
	M_{ij} & = & \langle  \nabla h_i\mathbf{f}, h_j \rangle = \int_{\Omega} \rho \nabla h_i\mathbf{f} h_j d\Omega; \nonumber \\
	B_{ij} & = & \langle  \nabla h_i\mathbf{b}, h_j \rangle = \int_{\Omega} \rho \nabla h_i\mathbf{b} h_j d\Omega; \nonumber \\
	P_{ij} & = & \langle  \nabla h_i\mathbf{g}, h_j \rangle = \int_{\Omega} \rho \nabla h_i\mathbf{g} h_j d\Omega.
\end{eqnarray}
Note that $M = B + \epsilon P$, and thus, it is possible to obtain matrix $M$ directly or by composition of its linear ($B$) and the non-linear ($P$) parts. This approximated system can be solved by obtaining the Schur decompostion $M = B + \epsilon P = WUW^{-1}$, where $U$ is an upper triangular matrix, and $W$ is a unitary matrix transformation. Then, we can define $\Phi = W^{-1}\mathbf{h}$, obtaining the following system of differential equations:
\begin{equation}
\left\{\begin{tabular}{l}
$\displaystyle\frac{d\Phi}{dt} = U\Phi$ \\
$\Phi(x_0) = W^{-1}\mathbf{h}(x_0)$
\end{tabular}\right. .
\end{equation}
which can be solved sequentially as seen in Section~\ref{sec:linear} to obtain $\mathbf{\Phi}$. Once $\mathbf{\Phi}$ is obtained, this solution can be transformed back into the basis functions $\mathbf{h} = W\mathbf{\Phi}$.

This process allows to directly obtain the solution of the system, treating the perturbed problem in the same way as the original. However, it does not allow to study the contribution of the perturbation in the system. Therefore, we introduce two approaches to overcome that limitation.

\subsection{Decomposition of linear and perturbing terms}

The objective of this methodology is to obtain the exact solution of the system provided by Eq.~\eqref{eq:hpert}. Let $\mathbf{h}$ be defined as the sum of two components, $\mathbf{h^m}$ which correspond to the solution of the non-perturbed problem, and $\epsilon\mathbf{h^p}$ which corresponds to the difference between the perturbed and non-perturbed problem, that is, $\mathbf{h} = \mathbf{h^m} + \epsilon\mathbf{h^p}$. Particularly, since $\mathbf{h^m}$ is the solution to the non-perturbed problem, we know that:
\begin{equation}\label{eq:hm}
\displaystyle\frac{d\mathbf{h^m}}{dt} = B\mathbf{h^m}.
\end{equation}
Moreover, by computing the derivative of $\mathbf{h}$ in terms of $\mathbf{h^m}$ and $\mathbf{h^p}$:
\begin{equation}
\displaystyle\frac{d\mathbf{h}}{dx} = \displaystyle\frac{d\mathbf{h^m}}{dx} + \epsilon\displaystyle\frac{d\mathbf{h^p}}{dx} = B\mathbf{h^m} + \epsilon B\mathbf{h^p} + \epsilon P\mathbf{h^m} + \epsilon^2 P\mathbf{h^p},
\end{equation}
and using the previous relation, the derivative of $\mathbf{h^p}$ is obtained:
\begin{equation}\label{eq:hp}
\epsilon\displaystyle\frac{d\mathbf{h^p}}{dt} = \left(B + \epsilon P\right)\epsilon\mathbf{h^p} + \epsilon P\mathbf{h^m}.
\end{equation}
Equations~\eqref{eq:hm} and~\eqref{eq:hp} can be combined to obtain the following system:
\begin{equation}
\displaystyle\frac{d}{dx}\left(\begin{tabular}{c}
$\mathbf{h^p}$ \\
$\mathbf{h^m}$
\end{tabular}\right) = \left(\begin{tabular}{c|c}
$\left(B + \epsilon P\right)$ & $P$ \\
\hline
0 & $B$
\end{tabular}\right)\left(\begin{tabular}{c}
$\mathbf{h^p}$ \\
$\mathbf{h^m}$
\end{tabular}\right),
\end{equation}
which is already a nearly upper triangular system. In fact, if we define $\Phi^p = W^{-1}\mathbf{h^p}$, and $\Phi^m = V^{-1}\mathbf{h^m}$, where $B + \epsilon P = WUW^{-1}$ and $B = VTV^{-1}$ are the Schur decompositions of $M$ and $B$ respectively:
\begin{eqnarray}
\displaystyle\frac{d\Phi^p}{dx} & = & U\Phi^p + W^{-1} P V \Phi^m, \nonumber \\
\displaystyle\frac{d\Phi^m}{dt} & = & T\Phi^m,
\end{eqnarray}
or in other words:
\begin{equation}
\displaystyle\frac{d}{dt}\left(\begin{tabular}{c}
$\Phi^p$ \\
$\Phi^m$
\end{tabular}\right) = \left(\begin{tabular}{c|c}
$U$ & $(W^{-1} P V)$ \\
\hline
0 & $T$
\end{tabular}\right)\left(\begin{tabular}{c}
$\Phi^p$ \\
$\Phi^m$
\end{tabular}\right),
\end{equation}
which is a triangular matrix that can be solved sequentially as in the non-perturbed case or in the direct approach. Finally, the solution of $\mathbf{h}$ is provided by:
\begin{equation}
\mathbf{h} = \mathbf{h^m} + \epsilon\mathbf{h^p} = V\Phi^m + \epsilon W\Phi^p.
\end{equation}
As can be seen, the final solution is a combination of the dynamic of the non-perturbed problem $\Phi^m$ with a contribution of the term related with the perturbation $\Phi^p$. This means that this approach requires to solve both the non-perturbed system (represented by $B$), and the complete perturbed system (represented by $M = B+\epsilon P$). Therefore this methodology is computational more expensive than a direct approach. Nevertheless, if the solution in $\Phi^m$ is already known, the methodology only requires to solve the contribution of $\Phi^p$, and it allows to determine the contributions of the perturbing term in the solution. 

The objective now is simplify this result by using the property that $\epsilon$ is a small parameter. This allows to obtain an approximated solution to Eq.~\eqref{eq:hpert} by only using the decomposition of the unperturbed system.

\subsection{Approximated perturbed model}

In this case we make use of the condition that $\epsilon \ll 1$ to simplify the problem. Particularly, we depart from Eqs.~\eqref{eq:hm} and~\eqref{eq:hp} and for convenience, we rewrite them in here:
\begin{eqnarray}
\epsilon\displaystyle\frac{d\mathbf{h^p}}{dx} & = & \left(B + \epsilon P\right)\epsilon\mathbf{h^p} + \epsilon P\mathbf{h^m}, \nonumber \\
\displaystyle\frac{d\mathbf{h^m}}{dx} & = & M\mathbf{h^m}.
\end{eqnarray}
Since $\epsilon \ll 1$, we can approximate the former system into this one:
\begin{eqnarray} \label{eq:decomposition_order}
\displaystyle\frac{d\mathbf{h^p}}{dx} & = & B\mathbf{h^p} + P\mathbf{h^m} + \mathcal{O}(\epsilon), \nonumber \\
\displaystyle\frac{d\mathbf{h^m}}{dx} & = & B\mathbf{h^m},
\end{eqnarray}
where $\mathcal{O}(\epsilon)$ represents the only term in the first expression ($\epsilon P\mathbf{h^p}$) that is affected by a power of $\epsilon$. Therefore, the previous system can be approximated in matrix notation as:
\begin{equation}
\displaystyle\frac{d}{dx}\left(\begin{tabular}{c}
$\mathbf{h^p}$ \\
$\mathbf{h^m}$
\end{tabular}\right) = \left(\begin{tabular}{c|c}
$B$ & $P$ \\
\hline
0 & $B$
\end{tabular}\right)\left(\begin{tabular}{c}
$\mathbf{h^p}$ \\
$\mathbf{h^m}$
\end{tabular}\right),
\end{equation}
which can be easily transformed into an upper triangular system by the use of the Schur decomposition from the unperturbed problem $B = VTV^{-1}$. In particular, if $\Phi^m = V^{-1}\mathbf{h^m}$, and $\Phi^p = V^{-1}\mathbf{h^p}$:
\begin{eqnarray}
\displaystyle\frac{d\Phi^p}{dx} & = & V^{-1}B\mathbf{h^p} + V^{-1}P\mathbf{h^m} = T\Phi^p +V^{-1}PV\Phi^m,  \nonumber \\
\displaystyle\frac{d\Phi^m}{dx} & = & V^{-1}B\mathbf{h^m} = T\Phi^m,
\end{eqnarray}
and thus:
\begin{equation}
\displaystyle\frac{d}{dx}\left(\begin{tabular}{c}
$\Phi^p$ \\
$\Phi^m$
\end{tabular}\right) = \left(\begin{tabular}{c|c}
$T$ & $(V^{-1} P V)$ \\
\hline
0 & $T$
\end{tabular}\right)\left(\begin{tabular}{c}
$\Phi^p$ \\
$\Phi^m$
\end{tabular}\right),
\end{equation}
which can be solved sequentially as in the rest of cases. As can be seen, the approximated system has the particularity of only requiring to study the non-perturbed problem in order to generate the upper triangular matrix. In that regard, it is important to note that the eigenvalues, and thus, the frequencies of the system, are now a linear combination of the ones from the unperturbed problem. In fact, this idea can be extended to higher orders of the perturbing term.


\subsection{Higher order perturbation model}

The objective now is to reduce the error associated with removing the terms in $\epsilon$ from Eq.~\eqref{eq:decomposition_order}. This is done by representing the evolution of the basis functions as a series expansion in terms of the small parameter $\epsilon$. Particularly, we define:
\begin{equation}
\mathbf{h} = \mathbf{h^{(0)}} + \sum_{i=1}^{\tau}\epsilon^i\mathbf{h^{(i)}},
\end{equation}
where $\tau$ is the order of the perturbation model and $\mathbf{h^{(i)}}$ are the solution components for each term of the series expansion. This implies that $\mathbf{h}^{(0)}$ represents the solution of the non-perturbed problem. By introducing this decomposition in the differential equation we obtain:
\begin{equation}
\displaystyle\frac{d\mathbf{h}}{dx} = \sum_{i=0}^{\tau}\epsilon^i\frac{d\mathbf{h^{(i)}}}{dx} = B\sum_{i=0}^{\tau}\epsilon^i\mathbf{h^{(i)}} + P \sum_{i=0}^{\tau}\epsilon^{i+1}\mathbf{h^{(i)}},
\end{equation}
which leads to:
\begin{eqnarray}
\displaystyle\frac{d\mathbf{h^{(0)}}}{dx} & = & B\mathbf{h^{(0)}}, \nonumber \\
\displaystyle\frac{d\mathbf{h^{(i)}}}{dx} & = & B\mathbf{h^{(i)}} + P\mathbf{h^{(i-1)}}, \nonumber \\
\displaystyle\frac{d\mathbf{h^{(\tau)}}}{dx} & = & B\mathbf{h^{(\tau)}} + P\mathbf{h^{(\tau-1)}} + \epsilon P\mathbf{h^{(\tau)}}.
\end{eqnarray}
This means that the matrix representation of the system can be approximated by:
\begin{equation}
\displaystyle\frac{d}{dx}\left(\begin{tabular}{c}
$\mathbf{h^{(\tau)}}$  \\
$\mathbf{h^{(\tau-1)}}$  \\
$\cdots$ \\
$\mathbf{h^{(1)}}$ \\
$\mathbf{h^{(0)}}$
\end{tabular}\right) = \left(\begin{tabular}{c|c|c|c|c}
$B$ & $P$ & $0$ & $\cdots$ & $0$\\
\hline
0 & $B$ & $P$ & $0$ & $\cdots$ \\
\hline
0 & $0$ & $\cdots$ & $\cdots$ & $0$ \\
\hline
0 & $0$ & $0$ & $B$ & $P$ \\
\hline
0 & $0$ & $0$ & $0$ & $B$
\end{tabular}\right)\left(\begin{tabular}{c}
$\mathbf{h^{(\tau)}}$  \\
$\mathbf{h^{(\tau-1)}}$  \\
$\cdots$ \\
$\mathbf{h^{(1)}}$ \\
$\mathbf{h^{(0)}}$
\end{tabular}\right).
\end{equation}
It is important to note that the term that is simplified in the equations is $P\epsilon^{\tau+1}\mathbf{h}^{(\tau)}$, and thus, much smaller than the main term $\mathbf{h}^{(0)}$. This property allows an improvement on the accuracy of the solution when compared with the previous approach. Then, if $\Phi^{(i)} = V^{-1}\mathbf{h^{(i)}}$, the former equation is transformed into:
\begin{equation}
\displaystyle\frac{d}{dx}\left(\begin{tabular}{c}
$\Phi^{(\tau)}$  \\
$\Phi^{(\tau-1)}$  \\
$\cdots$ \\
$\Phi^{(1)}$ \\
$\Phi^{(0)}$
\end{tabular}\right) = \left(\begin{tabular}{c|c|c|c|c}
$T$ & $V^{-1} P V$ & $0$ & $\cdots$ & $0$\\
\hline
0 & $T$ & $V^{-1} P V$ & $0$ & $\cdots$ \\
\hline
0 & $0$ & $\cdots$ & $\cdots$ & $0$ \\
\hline
0 & $0$ & $0$ & $T$ & $V^{-1} P V$ \\
\hline
0 & $0$ & $0$ & $0$ & $T$
\end{tabular}\right)\left(\begin{tabular}{c}
$\Phi^{(\tau)}$  \\
$\Phi^{(\tau-1)}$  \\
$\cdots$ \\
$\Phi^{(1)}$ \\
$\Phi^{(0)}$
\end{tabular}\right),
\end{equation}
which is solvable sequentially. Therefore, the solution to the system is provided by:
\begin{equation}
\mathbf{h} = \sum_{i=0}^{\tau}\epsilon^i\mathbf{h^{(i)}} = V\sum_{i=0}^{\tau}\epsilon^i\Phi^{(i)},
\end{equation}
and thus:
\begin{equation}
\mathbf{y} = H\mathbf{h} = HV\sum_{i=0}^{\tau}\epsilon^i\Phi^{(i)}.
\end{equation}

\subsection{Perturbation model with multiple sources}

In some cases, it might be of interest the study of perturbed systems where the perturbation is produced by different sources. Additionally, each source could contribute to the solution with a perturbation intensity that can range over a wide variety of orders of magnitude. For these reasons, in this subsection we present a more general perturbation model based on these assumptions.

Let a system of differential equations be represented as:
\begin{equation}
\left\{\begin{tabular}{l}
$\displaystyle\frac{d\mathbf{y}}{dx}=\mathbf{b}(\mathbf{y}) + \sum_{s=1}^{n_s}\sum_{k=1}^{n_k}\epsilon_s^k\mathbf{g^{(sk)}}(\mathbf{y}) $ \\
$\mathbf{y}(x_0) = \mathbf{y_0}$
\end{tabular}\right.,
\end{equation}
where $n_s$ is the number of different sources of perturbation, $n_k$ is the number of different orders of magnitude between the different sources, $\mathbf{g^{(sk)}}$ are the functions related to the perturbations, and $\epsilon_s\ll 1$ $\forall s$ are the small parameters related to the intensity of the perturbation. In this notation $\epsilon_s$ is defined such that $\epsilon_i \sim \epsilon_j$ $\forall i,j$, that is, the perturbation sources are classified based on the order of magnitude of their intensity. Note that based on the perturbation, there could be terms $\mathbf{g^{(sk)}}$ equal to zero. For instance, let the differential equation be:
\begin{equation}
	\displaystyle\frac{dy}{dx} = y + 0.1 y^2 + 0.02 y^3,
\end{equation}
then, $n_s = 2$, $n_k = 2$, $b(y) = y$, $\epsilon_1^1 = 0.1$, $g^{(11)} = y^2$, $\epsilon_1^2 = 0.01$, $g^{(12)} = 0$, $\epsilon_2^1 = \sqrt{0.02}$, $g^{(21)} = 0$, $\epsilon_2^2 = 0.02$, and $g^{(22)} = y^3$. As can be seen, $\epsilon_2$ is defined such that its order of magnitude is the same as $\epsilon_1$.

As in previous cases, we can define a set of orthonormal basis functions and obtain their derivatives:
\begin{equation}
\displaystyle\frac{d\mathbf{h}}{dx} = \nabla\mathbf{h}\mathbf{b} + \sum_{s=1}^{n_s}\sum_{k=1}^{n_k}\epsilon_s^k\nabla\mathbf{h}\mathbf{g^{(sk)}},
\end{equation}
and project them into the basis functions $\mathbf{h}$ to obtain a linearized system using the Garlekin method:
\begin{equation}
\left\{\begin{tabular}{l}
$\displaystyle\frac{d\mathbf{h}}{dx} = B\mathbf{h} + \sum_{s=1}^{n_s}\sum_{k=1}^{n_k}\epsilon_s^k P^{(sk)}\mathbf{h}$ \\
$\mathbf{h}(x_0) = \mathbf{h}(\mathbf{y_0})$
\end{tabular}\right.,
\end{equation}
where $P^{(sk)}$ is defined as the following inner product:
\begin{equation}
P_{ij}^{(sk)} = \langle  \nabla h_i\mathbf{g^{(sk)}}, h_j \rangle = \int_{\Omega} \rho \nabla h_i\mathbf{g^{(sk)}} h_j d\Omega.
\end{equation} 

The process now is to perform a decomposition of $\mathbf{h}$ using a series expansion of the small parameter. However, in this case we do not have a common small parameter affecting the differential equation. For that reason, we introduce the small parameter $\delta$ as the reference parameter for the expansion such that $\delta$ is able to represent all the orders of magnitude resulting in the terms $\epsilon_s^k$, therefore $\delta \sim \epsilon_s$  $\forall s$. That way, $\mathbf{h}$ can be decomposed into:
\begin{equation}
\mathbf{h} = \mathbf{h^{(0)}} + \sum_{i=1}^{\tau}\delta^i\mathbf{h^{(i)}},
\end{equation}
and introduced in the differential equation to obtain:
\begin{eqnarray}
\displaystyle\frac{d\mathbf{h}}{dx} & = & \frac{d\mathbf{h^{(0)}}}{dx} +  \sum_{i=1}^{op}\delta^i \frac{d\mathbf{h^{(i)}}}{dx} = L\mathbf{h^{(0)}} + \sum_{s=1}^{n_s}\sum_{k=1}^{n_k}\epsilon_s^k P^{(sk)}\mathbf{h^{(0)}} + \nonumber \\ 
& + & L\sum_{i=1}^{\tau}\delta^i\mathbf{h^{(i)}} + \sum_{s=1}^{n_s}\sum_{k=1}^{n_k}\epsilon_s^k P^{(sk)} \sum_{i=1}^{\tau}\delta^i\mathbf{h^{(i)}},
\end{eqnarray}
which leads to the following system:
\begin{eqnarray}
\displaystyle\frac{d\mathbf{h^{(0)}}}{dx} & = & B\mathbf{h^{(0)}} \nonumber \\
\displaystyle\frac{d\mathbf{h^{(1)}}}{dx} & = & B\mathbf{h^{(1)}} + \sum_{s=1}^{n_s}\frac{\epsilon_s}{\delta} P^{(s1)}\mathbf{h^{(0)}} \nonumber \\
\displaystyle\frac{d\mathbf{h^{(2)}}}{dx} & = & B\mathbf{h^{(2)}} + \sum_{s=1}^{n_s}\left(\frac{\epsilon_s}{\delta}\right)^2 P^{(s2)}\mathbf{h^{(0)}} + \sum_{s=1}^{n_s}\left(\frac{\epsilon_s}{\delta}\right) P^{(s1)}\mathbf{h^{(1)}} \nonumber \\
\displaystyle\frac{d\mathbf{h^{(i)}}}{dx} & = & B\mathbf{h^{(i)}} + \sum_{s=1}^{n_s}\sum_{k=1}^{i}\left(\frac{\epsilon_s}{\delta}\right)^k P^{(sk)} \mathbf{h^{(i-k)}} \nonumber \\
\displaystyle\frac{d\mathbf{h^{(\tau)}}}{dx} & = & B\mathbf{h^{(\tau)}} + \sum_{s=1}^{n_s}\sum_{k=1}^{i}\left(\frac{\epsilon_s}{\delta}\right)^k P^{(sk)} \mathbf{h^{(i-k)}} + \mathcal{O}(\delta).
\end{eqnarray}
In order to simplify the notation, we define $Q^{(ij)}$ as:
\begin{equation}
Q^{(i,j)} = \sum_{s=1}^{n_s}\left(\frac{\epsilon_s}{\delta}\right)^j P^{(sj)},
\end{equation}
then, the term $i$ of the system of differential equations becomes:
\begin{equation}
\displaystyle\frac{d\mathbf{h^{(i)}}}{dt} = B\mathbf{h^{(i)}} + \sum_{k=1}^{i}Q^{(i,k)} \mathbf{h^{(i-k)}} = B\mathbf{h^{(i)}} + \sum_{j=0}^{i-1}Q^{(i,i-j)} \mathbf{h^{(j)}}
\end{equation}
and so, the whole system can be approximated as:
\begin{equation}
\displaystyle\frac{d}{dx}\left(\begin{tabular}{c}
$\mathbf{h^{(\tau)}}$  \\
$\mathbf{h^{(\tau-1)}}$  \\
$\cdots$ \\
$\mathbf{h^{(1)}}$ \\
$\mathbf{h^{(0)}}$
\end{tabular}\right) = \left(\begin{tabular}{c|c|c|c|c}
$B$ & $Q^{(\tau,1)}$ & $Q^{(\tau,2)}$ & $\cdots$ & $Q^{(\tau,\tau)}$\\
\hline
0 & $B$ & $Q^{(\tau-1,1)}$ & $\cdots$ & $Q^{(\tau-1,\tau-1)}$ \\
\hline
$\cdots$ & $\cdots$ & $\cdots$ & $\cdots$ & $\cdots$ \\
\hline
0 & $0$ & $0$ & $B$ & $Q^{(1,1)}$ \\
\hline
0 & $0$ & $0$ & $0$ & $B$
\end{tabular}\right)\left(\begin{tabular}{c}
$\mathbf{h^{(\tau)}}$  \\
$\mathbf{h^{(\tau-1)}}$  \\
$\cdots$ \\
$\mathbf{h^{(1)}}$ \\
$\mathbf{h^{(0)}}$
\end{tabular}\right).
\end{equation}
By doing the unitary transformation $\Phi^{(i)} = V^{-1}\mathbf{h^{(i)}}$, such that $B=VTV^{-1}$, and naming $R^{(ij)}$ to:
\begin{equation}
R^{(i,j)} = V^{-1} Q^{(i,j)} V,
\end{equation}
the former equation can be rewritten as:
\begin{equation}
\displaystyle\frac{d}{dx}\left(\begin{tabular}{c}
$\Phi^{(\tau)}$  \\
$\Phi^{(\tau-1)}$  \\
$\cdots$ \\
$\Phi^{(1)}$ \\
$\Phi^{(0)}$
\end{tabular}\right) = \left(\begin{tabular}{c|c|c|c|c}
$T$ & $R^{(\tau,1)}$ & $R^{(\tau,2)}$ & $\cdots$ & $R^{(\tau,\tau)}$\\
\hline
0 & $T$ & $R^{(\tau-1,1)}$ & $\cdots$ & $R^{(\tau-1,\tau-1)}$ \\
\hline
$\cdots$ & $\cdots$ & $\cdots$ & $\cdots$ & $\cdots$ \\
\hline
0 & $0$ & $0$ & $T$ & $R^{(1,1)}$ \\
\hline
0 & $0$ & $0$ & $0$ & $T$
\end{tabular}\right)\left(\begin{tabular}{c}
$\Phi^{(\tau)}$  \\
$\Phi^{(\tau-1)}$  \\
$\cdots$ \\
$\Phi^{(1)}$ \\
$\Phi^{(0)}$
\end{tabular}\right),
\end{equation}
which can be solved sequentially. Therefore, the solution to the system is provided by:
\begin{equation}
\mathbf{h} = \sum_{i=0}^{\tau}\delta^i\mathbf{h^{(i)}} = V\sum_{i=0}^{\tau}\delta^i\Phi^{(i)},
\end{equation}
and thus:
\begin{equation}
\mathbf{y} = H\mathbf{h} = HV\sum_{i=0}^{\tau}\delta^i\Phi^{(i)}.
\end{equation}



\section{Examples of application}

In this section two examples of application are presented to show the performance and limitations of the proposed methodologies when dealing with non-linear perturbed problems. 

\subsection{Duffing oscillator}

The evolution of a simple duffing oscillator can be described by the following differential equation:
\begin{eqnarray}
	\displaystyle\frac{d q}{dt} & = & p, \nonumber \\
	\displaystyle\frac{d p}{dt} & = & -q - \epsilon q^3,
\end{eqnarray}
where $q$ is the position, $p$ is the velocity of the system, $\epsilon$ is the small parameter, and $t$ is the time evolution of the system. This means that the input for the algorithm is: $F(1,1,1) = 1$, $F(1,1,3) = 1$, $F(2,1,1) = -1$, $F(2,1,2) = 1$ for the linear part, and $F(2,2,1) = -\epsilon$, $F(2,2,2) = 3$ for the nonlinear term of the equation.

\begin{figure}[!ht]
	\centering
	{\includegraphics[width = 0.9\textwidth]{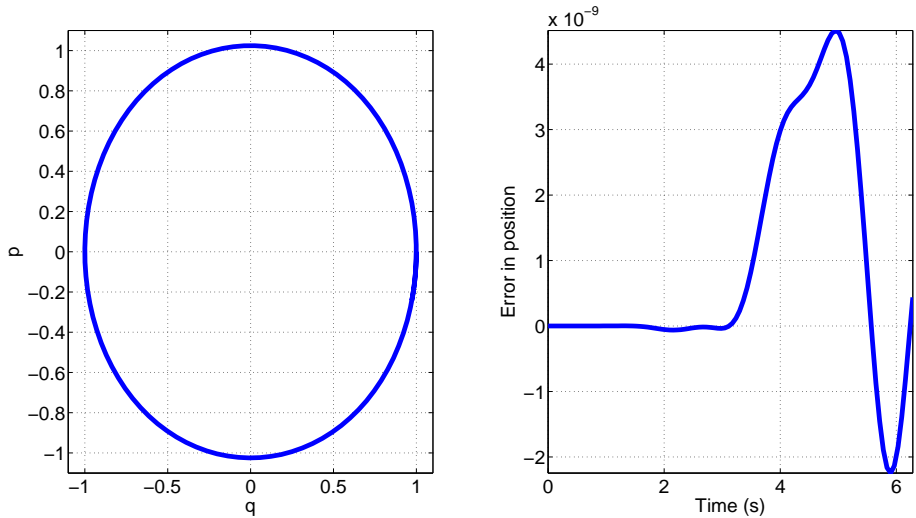}}
	\caption{State space solution (left) and position error (right) for a duffing oscillator with $\epsilon = 0.1$ using basis functions of order 11.}
	\label{fig:duffing_1_11}
\end{figure} 

Figure~\ref{fig:duffing_1_11} shows the system evolution and its error for an application of up to $\sigma = 11$ order basis functions to a differential equation with $\epsilon = 0.1$ and initial conditions $q_0 = 1$, and $p_0 = 0$. As can be seen, the error of this solution is in the order of magnitude of $10^{-9}$. If required, this error can be improved even further by just increasing the order of the basis functions $\sigma$. Conversely, if a solution with a lower number of terms is required, the order of the basis functions can also be reduced. This will decrease the accuracy of the solution, but it will also decrease the computational time to obtain it. Note also that the computation of the operator matrix $M$ is a one time computation effort for each order of basis functions. This means that once it is computed, it can be used for any combination of initial conditions.

\begin{figure}[!ht]
	\centering
	{\includegraphics[width = 0.9\textwidth]{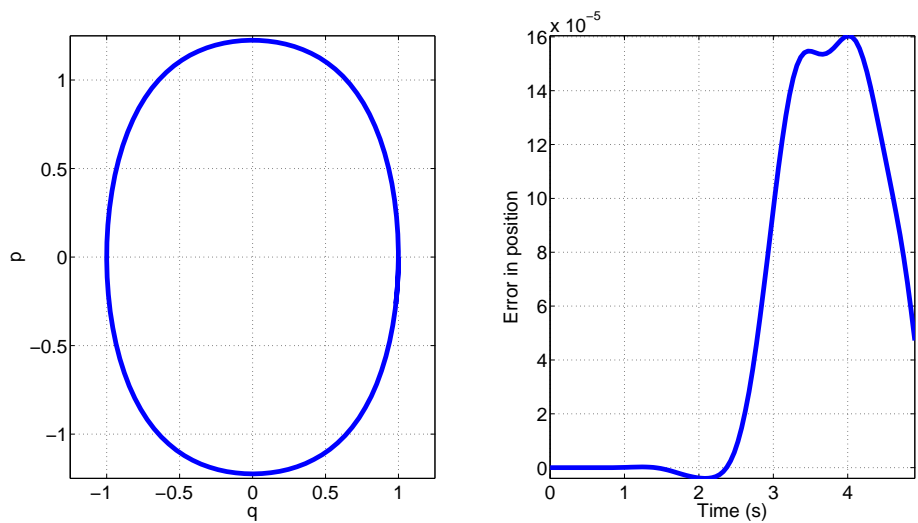}}
	\caption{State space solution (left) and position error (right) for a duffing oscillator with $\epsilon = 1$ using basis functions of order 11.}
	\label{fig:duffing_0_11}
\end{figure}

We can also study a more extreme of application. If we define $\epsilon = 1$, then, the non-linear part of the differential equation has the same intensity as the linear part. Figure~\ref{fig:duffing_0_11} shows the solution and associate position error when using basis functions of order 11 for this example. As can be seen, the error has increased when compared with the previous example to $10^{-4}$, however, the algorithm is still able to obtain an accurate solution to the problem. Another interesting thing to note is that the solution in $p$ leaves the region of definition of the Legendre polynomials while still providing a good accuracy. This is a situation that may happen depending on the problem due to the fact that the Garlekin method minimizes the error in the range of definition of the Legendre polynomials, and thus, sometimes the approximation can be extended out of the defined domain.

\begin{figure}[!ht]
	\centering
	{\includegraphics[width = 0.9\textwidth]{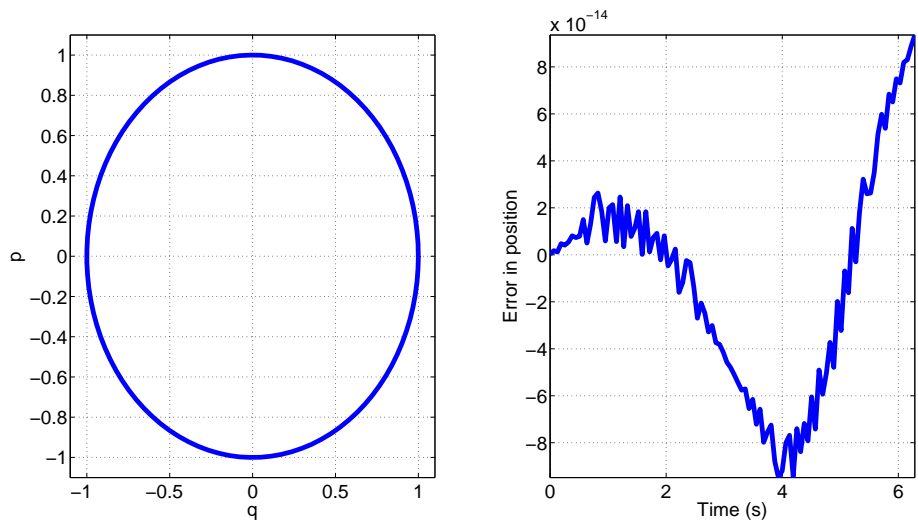}}
	\caption{State space solution (left) and position error (right) for a duffing oscillator with $\epsilon = 0.001$ using basis functions of order 11.}
	\label{fig:duffing_3_7}
\end{figure} 

If instead, we reduce the value of the small parameter $\epsilon$, the accuracy improves significantly. Particularly, if $\epsilon = 0.001$ it is possible to obtain nearly machine error precision with basis functions of order 7 as seen in Fig.~\ref{fig:duffing_3_7}. Note also that changing the initial condition $q_0$ and changing the value of the small parameter $\epsilon$ are equivalent operations since it is always possible to normalize the variables such that they range in $[-1,1]$. Particularly, if the maximum value of the variables is $Y$, we can perform the change of variable $r = q/Y$ and $s=p/Y$ to obtain this equivalent system:
\begin{eqnarray}
	\displaystyle\frac{d r}{dt} & = & s, \nonumber \\
	\displaystyle\frac{d s}{dt} & = & -r - \epsilon Y^2 r^3,
\end{eqnarray}
which has the same effect to having modified the value of the small parameter.

\begin{figure}[!ht]
	\centering
	{\includegraphics[width = 0.9\textwidth]{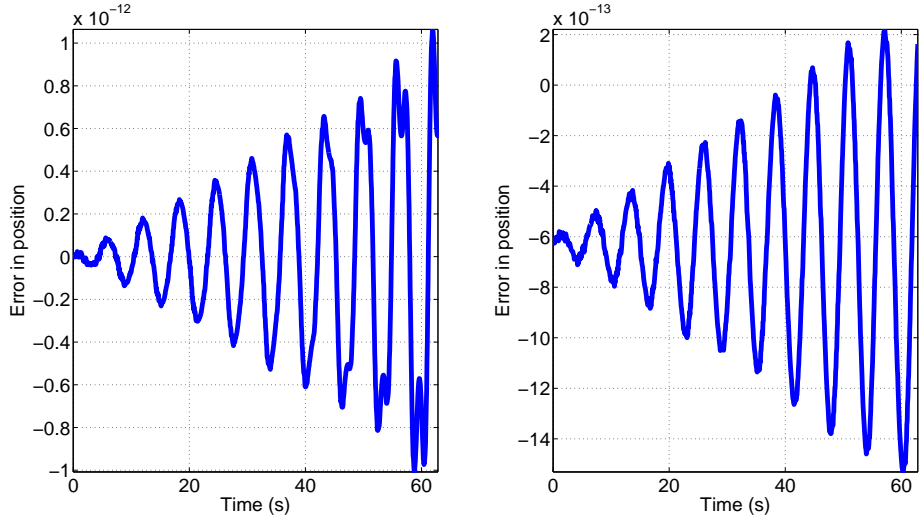}}
	\caption{Long term evolution of the position error for a duffing oscillator with $\epsilon = 0.001$ using basis functions of order 9 (left), and 11 (right).}
	\label{fig:duffing_longterm}
\end{figure}

Finally, it is also of interest to study the long term evolution of these solutions as well as the effect of using different order of basis functions. Figure~\ref{fig:duffing_longterm} shows the long term solution for $\epsilon = 0.001$ when using basis functions of order 9 and 11. As can be seen, the maximum error after nearly 11 complete rotations is $10^{-12}$ and $10^{-13}$ respectively, showing that the error improves as the order of basis functions increases. Moreover, it can be observed that the error slightly increases as the system evolves. This is common with other analytical and numerical perturbation techniques. In this regard, note also that the error presents an oscillating behavior. This is caused by the approximation errors in the frequencies obtained to represent the system.

In general, problems like the duffing oscillator with $\epsilon\ll 1$ are the kind of problems that are more suited for the application of the proposed algorithm for non-linear systems, specially in systems with quasi-periodic solutions. Examples of this kind of systems can be seen, for instance, in astrodynamics. In particular, Refs.~\cite{koopman_zonal,koopman_analysis} show the application and study of this methodology to the zonal harmonics problem of an object orbiting an oblate celestial body. On the other hand, Ref.~\cite{koopman_3body} shows an application to the three-body problem in astrodynamics and how to use the operator matrix $M$ to create a control scheme. This shows the potential use of this algorithm to study complex non-linear perturbed systems.

\subsection{Van der Pole oscillator}

Another interesting example of application is the Van der Pole oscillator due to its unstable nature close to the equilibrium point, the presence of the limit cycle, and its eigenvalues with real positive part. The dynamics of this oscillator can be represented by:
\begin{eqnarray}
	\displaystyle\frac{d q}{dt} & = & p, \nonumber \\
	\displaystyle\frac{d p}{dt} & = & -q + \epsilon\left(1-q^2\right)p,
\end{eqnarray}
where $q$ is the position, and $p$ is the velocity of the system. This implies an algorithm input in the form: $F(1,1,1) = 1$, $F(1,1,3) = 1$; $F(2,1,1) = -1$, $F(2,1,2) = 1$; $F(2,2,1) = \epsilon$, $F(2,2,3) = 1$; and $F(2,3,1) = -\epsilon$, $F(2,3,2) = 2$, $F(2,3,3) = 1$.

As a first example of application, we impose the conditions $\epsilon = 0.1$, $q_0 = 0.7$, and $p_0 = 0$. This places the dynamics of the oscillator in the region where the particle is going to move away from the equilibrium point of the dynamic. Figure~\ref{fig:pole_1_11} shows the state space and the position error of the algorithm when using basis functions of order 11. As can be seen, the error is worse than in the duffing oscillator. This is caused by the fact that the duffing oscillator has a quasi-periodic dynamic which is easier to represent using operator methodology. Nevertheless, the algorithm is able to obtain a good accuracy even in this problem. 

\begin{figure}[!ht]
	\centering
	{\includegraphics[width = 0.8\textwidth]{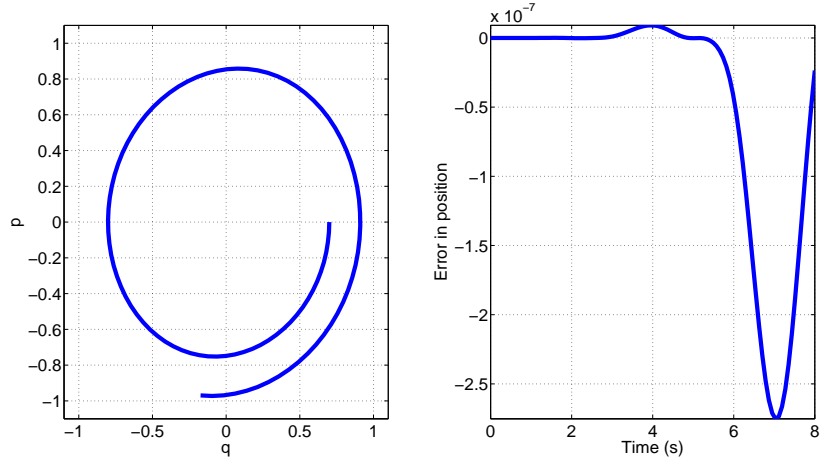}}
	\caption{State space solution (left) and position error (right) for the Van der Pole oscillator using basis functions of order 11.}
	\label{fig:pole_1_11}
\end{figure}

\begin{figure}[!ht]
	\centering
	{\includegraphics[width = 0.8\textwidth]{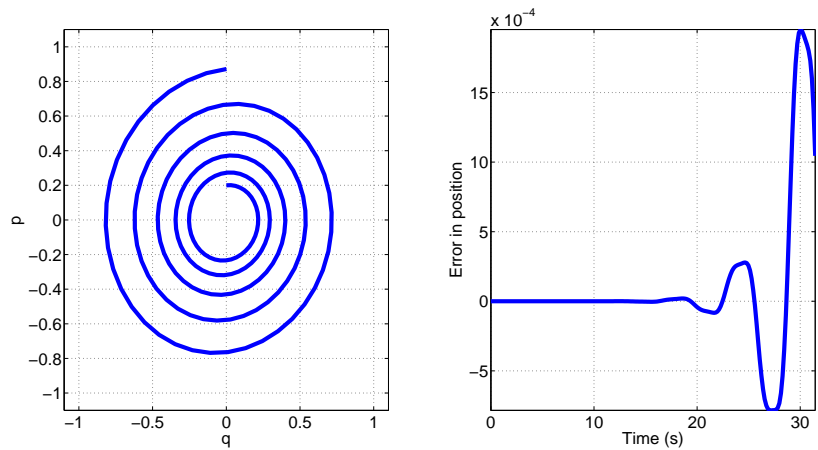}}
	\caption{State space solution (left) and position error (right) for the Van der Pole oscillator close to the equilibrium point using basis functions of order 15.}
	\label{fig:pole_1_15}
\end{figure} 

In the Van der Pole oscillator, it is of special interest the study of the dynamic when the particle is closer to the equilibrium point and its error after several rotations. To that end, Fig.~\ref{fig:pole_1_15} shows the evolution and position error when performing a propagation using basis functions of order 15 with initial conditions $\epsilon = 0.1$, $q_0 = 0.0$, and $p_0 = 0.2$. Particularly, the first rotation has an error smaller than $10^{-9}$ which rapidly increases over the different revolutions. As in the duffing oscillator, this happens primarily due to the approximation performed in the eigenvalues of the system, which for the Van der Pole oscillator, also affect the magnitude of the oscillation (due to having positive real eigenvalues even on the unperturbed system. Nevertheless, the methodology is able to represent the dynamics of the problem but with a smaller accuracy, and shows that this methodology converges faster in problem whose unperturbed system has no positive real eigenvalues, making the solution numerically more stable.

Finally, it is also possible to study the region close to the limit circle. However, since this region is out of the domain of definition of the Legendre polynomials, it is first necessary to normalize the variables. Particularly, if the maximum value of the variables is $Y = 2$, we can perform the change of variable $r = q/Y$ and $s=p/Y$ to obtain this equivalent system:
\begin{eqnarray}
	\displaystyle\frac{d r}{dt} & = & s, \nonumber \\
	\displaystyle\frac{d s}{dt} & = & -r + \epsilon\left( 1 - Y^2 r^2\right)s.
\end{eqnarray}

Figures~\ref{fig:pole_inside} and~\ref{fig:pole_outside} shows the evolution and position error when applying $\epsilon = 0.1$ and initial conditions, $q_0 = 0$, and $p_0 = 1.95$ and $p_0 = 2.05$ respectively. Both figures represent the two potential situations close to the limit circle, Fig~\ref{fig:pole_inside} coming from the interior, and Fig.~\ref{fig:pole_outside} coming from the outside. As can be seen, the error is now smaller and more comparable with the one obtained for the duffing oscillator (note that the effective value of the small parameter has increased to 0.4 due to the normalization). This is due to the more periodic behavior that the system is presenting. In fact, if the periodic initial condition is imposed $p_0 = 2$, the error obtained further decreases, showing the increase in performance expected when dealing with periodic and quasi-periodic solutions.

\begin{figure}[!ht]
	\centering
	{\includegraphics[width = 0.8\textwidth]{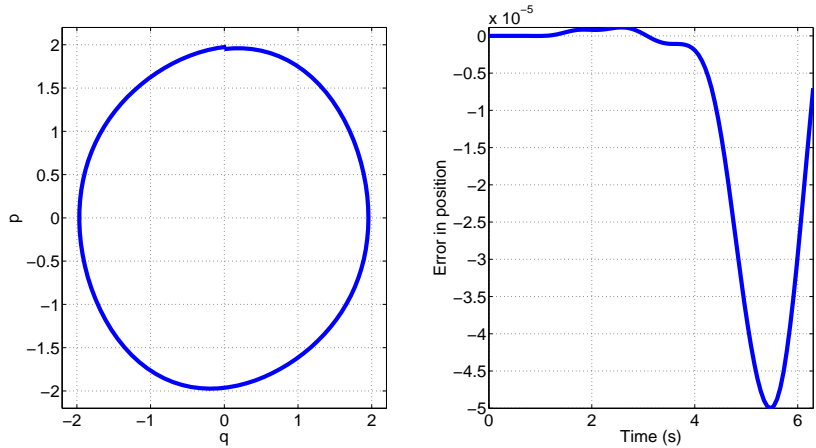}}
	\caption{State space solution (left) and position error (right) for the Van der Pole oscillator using basis functions of order 11 close to the limit circle (inside dynamic).}
	\label{fig:pole_inside}
\end{figure} 

\newpage
\begin{figure}[!ht]
	\centering
	{\includegraphics[width = 0.8\textwidth]{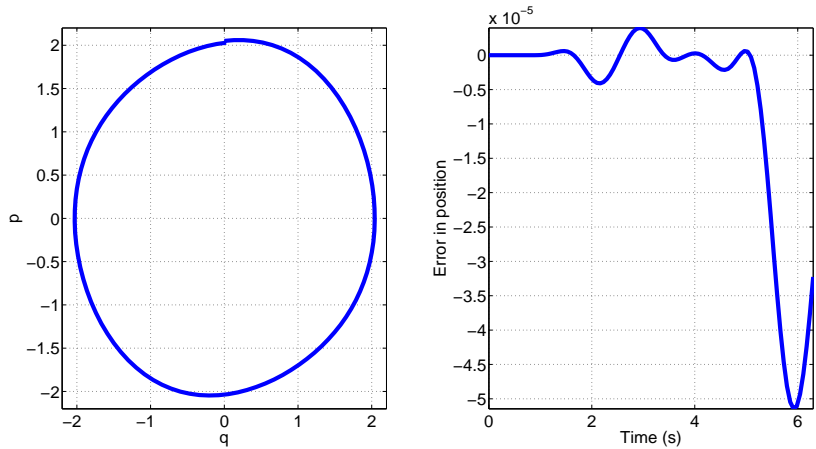}}
	\caption{State space solution (left) and position error (right) for the Van der Pole oscillator using basis functions of order 11 close to the limit circle (outside dynamic).}
	\label{fig:pole_outside}
\end{figure} 


\section{Conclusions}

This work focuses on the generation of solutions to ordinary differential equations. Particularly, and for linear systems, this work makes use of the Schur decomposition to perform unitary transformations of the original variables of the differential equation such that the resultant system is upper triangular. Once this is done, a theorem and an algorithm are introduced to sequentially solve the upper triangular differential equation analytically.

This result is later applied to the solution of perturbed non-linear systems using operator theory, with special focus on perturbation problems with a strong linear component. To that end, a general methodology is presented for the use of any kind of orthonormal basis functions. Later, the specific algorithm for the use of Legendre polynomials is presented, which allows to analytically obtain the approximated linearization of a polynomial differential equation in closed form. This allows to obtain any order of the solution following an automatic process. Additionally, and since the linearization does not depend on the initial conditions, the linearization only has to be computed once for a given set of basis functions.

These non-linear methodologies allow to solve problems where the non-linear term is weaker than the linear part, and they are specially suited for problems whose solution is quasi-periodic, making these algorithms very interesting for complex perturbation problems in multiple variables. Moreover, and compared with the direct application of the Koopman operator, this methodology can be applied to non-diagonalizable systems. This allows to apply the technique to a wider set of problems, and at the same time, it enables the use of a wider set of basis functions to represent the solution. 

Additionally, a perturbation model based on the Schur decomposition is presented. This allows to separate the contributions of the linear and perturbed terms of the differential equation. Moreover, and when approximating the expansion in the small parameter, this methodology allows to obtain the evolution of a perturbed problem using the solution and Schur decomposition of the unperturbed problem, which greatly reduces the complexity of the computation.

Finally, several examples of application are presented where the accuracy performance and potential limitations of these techniques are shown. In that regard, it is important to note that the solution provided by these methodologies is an analytical solution, and thus, it allows to perform the spectral study of the system. This has many applications for different problems, specially for control applications.



\begin{thebibliography}{}

	\bibitem{book} \textsc{Nayfeh, A. H.}, and \textsc{Mook, D. T.}, \textit{Nonlinear oscillations}, Wiley, 1980. doi: 10.1002/9783527617586.
	
	\bibitem{var} \textsc{He, J.H.}, \textit{Variational iteration method for autonomous ordinary differential systems}, Applied Mathematics and Computation, Vol. 114, No. 2--3, 2000, pp. 115--123. doi: 10.1016/S0096-3003(99)00104-6.
	
	\bibitem{jorba} \textsc{Jorba, A.}, and \textsc{Zou, M.}, \textit{A software package for the numerical integration of ODEs by means of high-order Taylor methods}, Experimental Mathematics, Vol. 14, No. 1, 2005, pp. 99-117. doi: 10.1080/10586458.2005.10128904.
	
	\bibitem{TFC} \textsc{Mortari, D.} \textit{Least-Squares Solution of Linear Differential Equations}, Mathematics, Vol. 5, No. 48, 2017. doi: 10.3390/math5040048 .
	
	\bibitem{schur} \textsc{Golub, G.H.}, and \textsc{Van Loan, C.F.}, \textit{Matrix Computations}, Johns Hopkins University Press, 1983. ISBN: 0-8018-5414-8.
	
	\bibitem{conway2019course} \textsc{Conway, J. B.} \textit{A course in functional analysis}, Springer, Vol. 96, 2019. doi: 10.1007/978-1-4757-3828-5.
	
	\bibitem{kowalski1991nonlinear} \textsc{Kowalski, K.}, and \textsc{Steeb, W.} \textit{Nonlinear dynamical systems and Carleman linearization}, World Scientific, 1991. ISBN: 978-9810205874.
	
	\bibitem{naylor2000} \textsc{Naylor, A. W.}, and \textsc{Sell, G. R.} \textit{Linear Operator Theory in Engineering and Science}, Applied Mathematical Sciences, Springer, 2000. ISBN: 9780387950013.
	
	\bibitem{prigogine} \textsc{Prigogine, I.} \textit{Non--equilibrium statistical mechanics}, Dover Publications, 1962. ISBN: 978-0486815558.

	\bibitem{williams2015data} \textsc{Williams, M. O., Kevrekidis, I. G.}, and \textsc{Rowley, C. W.}, 
	\textit{A data--driven approximation of the koopman operator: Extending dynamic mode decomposition}, Journal of Nonlinear Science, Vol. 25, No. 6, 2015, pp. 1307--1346. doi : 10.1007/s00332-015-9258-5.
	
	\bibitem{mezic2013analysis} \textsc{Mezi{\'c}, I.},
	\textit{Analysis of fluid flows via spectral properties of the Koopman operator}, Annual Review of Fluid Mechanics, Vol. 45, 2013, pp. 357--378. doi: 10.1146/annurev-fluid-011212-140652.

	\bibitem{brunton2016koopman} \textsc{Brunton, S. L., Brunton, B. W., Proctor, J. L.}, and \textsc{Kutz, J. N.},
	\textit{Koopman invariant subspaces and finite linear representations of nonlinear dynamical systems for control}, Public Library of Science one, Vol. 11, No. 2, 2016, pp. e0150171. doi: 10.1371/journal.pone.0150171.
	
	\bibitem{surana2016koopman} \textsc{Surana, A.},
	\textit{Koopman operator based observer synthesis for control-affine nonlinear systems}, 2016 IEEE 55th Conference on Decision and Control (CDC), 2016, pp. 6492--6499. doi: 10.1109/CDC.2016.7799268.
	
	\bibitem{koopman1931hamiltonian} \textsc{Koopman, B. O.} \textit{Hamiltonian systems and transformation in Hilbert space}, Proceedings of the National Academy of Sciences, Vol. 17, No. 5, 1931, pp. 315--318. doi: 10.1073/pnas.17.5.315.	

	\bibitem{neumann1932operatorenmethode} \textsc{Neumann, J. v},
	\textit{Zur Operatorenmethode in der klassischen Mechanik}, Annals of Mathematics, 1932, pp. 587--642. doi: 10.2307/1968537.

	\bibitem{zonal} \textsc{Arnas, D.}, and \textsc{Linares, R.},
	\textit{A set of orbital elements to fully represent the zonal harmonics around an oblate celestial body}, Monthly Notices of the Royal Astronomical Society, Vol. 502, 2021, pp. 4247-42261. doi: 10.1093/mnras/staa4040.
	
	\bibitem{koopman_zonal} \textsc{Arnas, D.}, and \textsc{Linares, R.},
	\textit{Approximate analytical solution to the zonal harmonics problem using Koopman operator theory}, Journal of Guidance, Control, and Dynamics, Vol. 44, No. 11, 2021, pp. 1909–-1923. doi: 10.2514/1.G005864.
	
	\bibitem{koopman_analysis} \textsc{Arnas, D., Linares, R.}, and \textsc{Alfriend, K.},
	\textit{An analysis of Koopman-based perturbation theory applied to the motion about an oblate planet}, 31st AAS/AIAA Space Flight Mechanics Meeting, AAS/AIAA, 2021, pp. AAS 21–365.
	
	\bibitem{koopman_3body} \textsc{Servadio, S., Arnas, D.}, and \textsc{Linares, R.},
	\textit{Dynamics near the three-body libration points via the Koopman operator theory}, 2021 AAS/AIAA Astrodynamics Specialist Conference, AAS/AIAA, 2021, pp. AAS 21-688.

\end{thebibliography}
\end{document}